\documentclass[10pt]{article}
\usepackage{eurosym}
\usepackage{amssymb}
\usepackage{amsmath}
\usepackage{amsthm}
\usepackage{amsfonts}
\usepackage{mathrsfs}
\usepackage{graphicx,color}
\usepackage{xcolor}
\usepackage{bbm}

\usepackage{titling}
\usepackage{abstract}
\usepackage{titlesec}
\usepackage{abstract}

\newtheorem{theorem}{Theorem}
\theoremstyle{plain}

\newtheorem{lemma}{Lemma}

\newtheorem{proposition}{Proposition}
\newtheorem{remark}{Remark}

\numberwithin{equation}{section}

\usepackage{titlesec} 
\renewcommand\thesection{\arabic{section}} 
\renewcommand\thesubsection{\arabic{subsection}} 
\titleformat{\section}[block]{\large\scshape\centering}{\thesection.}{1em}{} 
\titleformat{\subsection}[block]{\large}{\thesection.\thesubsection.}{1em}{} 
\usepackage[top=5cm,bottom=5cm,left=4cm,right=4cm,twoside]{geometry}

\usepackage[T1]{fontenc}
\title{%
 \textbf{\textsc{ Singular quasilinear convective systems involving variable exponents}}}
\author{%
\textsc{Abdelkrim Moussaoui} \\[1ex] 
\normalsize Applied Mathematics Laboratory (LMA)\\ \normalsize Faculty of Exact Sciences and Biology departement\\ \normalsize  Faculty of Natural \& Life Sciences\\
\normalsize A. Mira Bejaia University, Targa Ouzemour, 06000 Bejaia, Algeria\\
\normalsize abdelkrim.moussaoui@univ-bejaia.dz \\[2ex] 
\textsc{Dany Nabab\textsuperscript{1} and Jean Velin\textsuperscript{2}} \\[1ex] 
\normalsize Departement of Mathematics and Informatic (DMI) \\\normalsize Laboratory LAMIA, Campus of Fouillole\\\normalsize Faculty of Exact and Natural Sciences \\ \normalsize University of Antilles, 97159 Pointe-\`{A}-Pitre,
Guadeloupe (FWI)\\
\normalsize \textsuperscript{1} dany.nabab@univ-antilles.fr\\
\normalsize\textsuperscript{2} jean.velin@univ-antilles.fr
}
\date{}
\usepackage{fancyhdr}
\begin{document}
\maketitle
\footnote{AMS Subject Classifications: 35J75; 35J48; 35J92}
\footnote{Keywords: $p(x)$-Laplacian, variable exponents, fixed point, singular system, gradient estimate, regularity.}
\vspace{-2cm}
\begin{abstract}
\noindent\textbf{Abstract.} The paper deals with the existence of solutions for quasilinear elliptic
systems involving singular and convection terms with variable exponents. Our
approach combines the sub-supersolutions method and Schauder's fixed point
theorem.
\end{abstract}

\section{Introduction}

Let $\Omega \subset {\mathbb{R}}^{N}$ $(N\geq 2)$ be a bounded domain with
smooth boundary $\partial \Omega $. Given $p_{i}\in C^{1}(\overline{\Omega }%
),$ $1<p_{i}^{-}\leq p_{i}^{+}<N$ with%
\begin{equation*}
\begin{array}{l}
p_{i}^{-}=\inf_{x\in \Omega }p_{i}(x)\text{ \ and \ }p_{i}^{+}=\sup_{x\in
\Omega }p_{i}(x),%
\end{array}%
\end{equation*}%
we deal with the following quasilinear elliptic system%
\begin{equation*}
(\mathrm{P})\qquad \left\{ 
\begin{array}{l}
-\Delta _{p_{i}(x)}u_{i}=f_{i}(x,u_{1},u_{2},\nabla u_{1},\nabla u_{2})\text{
\ in }\Omega \\ 
u_{i}>0\text{ \ in }\Omega ,\text{ \ }u_{i}=0\text{ \ on }\partial \Omega ,%
\text{ }i=1,2,%
\end{array}%
\right.
\end{equation*}%
where $-\Delta _{p_{i}(x)}$ stands for the $p_{i}(x)$-Laplacian differential
operator defined by 
\begin{equation*}
-\Delta _{p_{i}(x)}u_{i}=-div(|\nabla u_{i}|^{p_{i}(x)-2}\nabla u_{i}),\text{
for }u_{i}\in W_{0}^{1,p_{i}(x)}(\Omega ).
\end{equation*}%
The nonlinear terms $f_{1}(x,u_{1},u_{2},\nabla u_{1},\nabla u_{2})$ and $%
f_{2}(x,u_{1},u_{2},\nabla u_{1},\nabla u_{2})$ which are often expressed as
dealing with convection terms, are of Carath\'{e}odory type. Namely, for every $%
(s_{1},s_{2},\xi _{1},\xi _{2})\in (\mathbb{R}_{+}^{\star})^{2}\times \mathbb{R}^{2N}$, we assume
that $f_{i}(\cdot,s_{1},s_{2},\xi _{1},\xi _{2})$ is Lebesgue measurable in $\Omega$, and, for a.e. $x\in \Omega$, $f_{i}(x,\cdot ,\cdot ,\cdot ,\cdot )$ is
continuous in $(\mathbb{R}_{+}^{\star})^{2}\times \mathbb{R}^{2N}$. A solution of $(\mathrm{P})$ is
understood in the weak sense, that is, a pair $(u_{1},u_{2})\in
W_{0}^{1,p_{1}(x)}(\Omega )\times W_{0}^{1,p_{2}(x)}(\Omega )$ satisfying 
\begin{equation}
\int_{\Omega }|\nabla u_{i}|^{p_{i}(x)-2}\nabla u_{i}\nabla \varphi _{i}\
dx=\int_{\Omega }f_{i}(x,u_{1},u_{2},\nabla u_{1},\nabla u_{2})\varphi _{i}\
dx,  \label{7}
\end{equation}%
for all $\varphi _{i}\in W_{0}^{1,p_{i}(x)}(\Omega )$.

Our main purpose is to establish the existence and regularity of solutions
for quasilinear singular convective system $(\mathrm{P})$ satisfying the
assumption:

\begin{description}
\item[$($\textrm{H}$_{f})$] There exist constants $M_{i},m_{i}>0$, and
functions $\alpha _{i},\beta _{i},\gamma _{i},\bar{\gamma}_{i}\in C(%
\overline{\Omega })$, such that 
\begin{equation*}
m_{i}s_{1}^{\alpha _{i}(x)}s_{2}^{\beta _{i}(x)}\leq f_{i}(x,s_{1},s_{2},\xi
_{1},\xi _{2})\leq M_{i}(s_{1}^{\alpha _{i}(x)}s_{2}^{\beta
_{i}(x)}+\left\vert \xi _{1}\right\vert ^{\gamma _{i}(x)}+\left\vert \xi
_{2}\right\vert ^{\bar{\gamma}_{i}(x)}),
\end{equation*}%
for a.e. $x\in \Omega $, for all $s_{1},s_{2}>0$ and all $\xi _{1},\xi
_{2}\in 
\mathbb{R}
^{N},$ $i=1,2$.
\end{description}

The dependence of the right hand side terms on the solution and its gradient
deprives system $(\mathrm{P})$ of a variational structure. Thereby
variational methods are not applicable. Moreover, due to the presence of
convection terms, even the so called topological methods as
sub-supersolutions and fixed points technique cannot be directly implemented.
Another important feature in studying problem $(\mathrm{P})$ is that the
nonlinearities can exhibit singularities when the variables $u_{1}$ and $%
u_{2}$ approach zero. This occur through the following condition

\begin{description}
\item[$($\textrm{H}$_{\protect\alpha ,\protect\beta ,\protect\gamma })$] 
\begin{equation*}
|\alpha _{i}^{\mp }|+|\beta _{i}^{\mp }|<p_{i}^{-}-1,
\end{equation*}%
and 
\begin{equation*}
0\leq \min \{\gamma _{i}^{-},\bar{\gamma}_{i}^{-}\}\leq \max \{\gamma
_{i}^{+},\bar{\gamma}_{i}^{+}\}<p_{i}^{-}-1,
\end{equation*}%
where 
\begin{equation*}
\alpha _{i}^{\mp }:=\left\{ 
\begin{array}{ll}
\alpha _{i}^{-} & \text{if }\alpha _{i}(\cdot )>0 \\ 
\alpha _{i}^{+} & \text{if }\alpha _{i}(\cdot )<0%
\end{array}%
\right. \text{ }\beta _{i}^{\mp }:=\left\{ 
\begin{array}{ll}
\beta _{i}^{-} & \text{if }\beta _{i}(\cdot )>0 \\ 
\beta _{i}^{+} & \text{if }\beta _{i}(\cdot )<0%
\end{array}%
\right. ,\text{ }i=1,2.
\end{equation*}
\end{description}

Precisely, singularities appear in system whenever one of the exponents at
least is negative, that is, $\min \{\alpha _{1},\alpha _{2},\beta _{1},\beta
_{2}\}<0$. This represents a major hurdle to overcome. This difficulty is
heightened by the very emphasized singularity character of $(\mathrm{P})$
that stems from $($\textrm{H}$_{\alpha ,\beta ,\gamma })$ when $\alpha
_{i}^{\mp }+\beta _{i}^{\mp }<0$. In this case, hypothesis $($\textrm{H}$%
_{\alpha ,\beta ,\gamma })$ is strengthened by assuming

\begin{description}
\item[$($\textrm{\~{H}}$_{\protect\alpha ,\protect\beta ,\protect\gamma })$] %
If $\alpha _{i}^{\mp }+\beta _{i}^{\mp }<0$, we have%
\begin{equation*}
\left\vert \alpha _{i}^{\mp }\right\vert +\left\vert \beta _{i}^{\mp
}\right\vert \leq \frac{1}{Np_{i}^{\prime +}},
\end{equation*}%
and%
\begin{equation*}
0\le\gamma _{i}(x)\leq \frac{p_{1}(x)}{Np_{i}^{\prime }(x)}\text{ \ and \ }%
0\le\overline{\gamma }_{i}(x)\leq \frac{p_{2}(x)}{Np_{i}^{\prime }(x)},\text{
for }x\in \Omega .
\end{equation*}
\end{description}

Quasilinear convective system $(\mathrm{P})$ has been rarely investigated in
the literature. Actually, according to our knowledge, \cite{NV} is the only
paper that has addressed this issue in the regular case, that is when all
exponents are positive. Existence result is obtained applying the recent
topological degree of Berkovits. The virtually non-existent works devoted to
the singular case of convective systems is partly due to the involvement of
the $p_{i}(x)$-Laplacian operator. This fact results in the lack of
properties such as homogeneity making it highly challenging task to
establish a control on solutions and especially on their gradient. When $%
p_{i}(x)$ is reduced to be a constant, $-\Delta _{p_{i}(x)}$ becomes the
well-known $p_{i}$-Laplacian operator. In this respect, by relying on the a
priori gradient estimate in \cite{BE, CM}, the existence of solutions for
singular convective systems have been investigated in \cite{CLM, DM}. Still
in the context of constant exponents, for singular system $(\mathrm{P})$
defined in whole space $%
\mathbb{R}
^{N}$ we quote \cite{GMM} while for the case of Neumann boundary condition
we refer to \cite{GM}. We also mention \cite{AM2} focusing on a singular
system of type $(\mathrm{P})$ corresponding to the semilinear case, that is
when $p_{i}(x)=2$ ($i=1,2$).

When convection terms are canceled, singular system $(\mathrm{P})$ was
recently examined in \cite{AM, AMT}. In this context, depending on the sign
of $\alpha _{i}(\cdot )$ and $\beta _{i}(\cdot ),$ two complementary
structures for the system $(\mathrm{P})$ appear: cooperative and competitive
structure (see \cite{AM}). Here, the important structural disparity of the
latter makes nonlinearities $f_{1}$ and $f_{2}$ (without gradient terms)
behaving in a drastically different way. This fact has led in \cite{AMT} to
consider only the cooperative system involving logarithmic growth while in 
\cite{AM}, a separate study corresponding to each structure is required. We
emphasize that in the present work, neither cooperative nor competitive
structure on the system $(\mathrm{P})$ is imposed. In fact, these both
complementary structures for the system $(\mathrm{P})$ are handled
simultaneously without referring to them.

Our main result is stated as follows.

\begin{theorem}
\label{T1}Assume \textrm{(H}$_{f}$\textrm{)} holds. Then

$\mathbf{(i)}$ Under assumption $($\textrm{H}$_{\alpha ,\beta ,\gamma })$
with $\alpha _{i}^{\mp }+\beta _{i}^{\mp }>0$, system $(\mathrm{P})$ has a
bounded (positive) solution $\left( u_{1},u_{2}\right) $ in $C_{0}^{1,\tau }(%
\overline{\Omega })\times C_{0}^{1,\tau }(\overline{\Omega }),$ for certain $%
\tau \in (0,1),$ satisfying%
\begin{equation}
u_{i}(x)\geq c_{0}d(x),\text{ for a constant }c_{0}>0,\text{ }i=1,2,
\label{44}
\end{equation}

$\mathbf{(ii)}$ Under assumption $(\mathrm{\tilde{H}}_{\alpha ,\beta ,\gamma
})$, system $(\mathrm{P})$ admits a (positive) solution $\left(
u_{1},u_{2}\right) $ in $(W_{0}^{1,p_{1}(x)}(\Omega )\cap L^{\infty }(\Omega
))\times (W_{0}^{1,p_{2}(x)}(\Omega )\cap L^{\infty }(\Omega ))$ satisfying (%
\ref{44}).
\end{theorem}

Our approach is chiefly based on Schauder's fixed point theorem. In this
respect, comparison arguments as well as a priori estimates are crucial to
get the appropriate localization of the desired fixed point which is
actually solution of $(\mathrm{P})$. This is achieved through a control on
solutions and their gradient which in itself represents a significant
feature of our result. At this point, the choice of suitable functions with
an adjustment of adequate constants is crucial. However, this would not be
enough without making use of the new Mean Value theorem (cf. Theorem \ref%
{MVT2} in Appendix) that is decisive to offset the lack of homogeneity
property and to deal with the variable exponents attendance. It should be
noted that the Mean Value theorem is a key ingredient in getting the
gradient estimate thus generalizing that corresponding to the case of
constant exponents problems stated in \cite{BE, CM}.

The rest of the paper is organized as follows. Sections \ref{S2} and \ref{S3}
establish gradient estimates and a priori bounds; Section \ref{S4} deals
with comparison properties; Section \ref{S5} presents the proof of the main
result while Section \ref{S6} contains the new Mean Value Theorem.

\section{A priori estimates}

\label{S2}

Let $L^{p(x)}(\Omega )$ be the generalized Lebesgue space that consists of
all measurable real-valued functions $u$ satisfying 
\begin{equation*}
\rho _{p(x)}(u)=\int_{\Omega }|u(x)|^{p(x)}dx<+\infty ,
\end{equation*}%
endowed with the Luxemburg norm%
\begin{equation*}
\begin{array}{l}
\left\Vert u\right\Vert _{p(x)}=\inf \{\tau >0:\rho _{p(x)}(\frac{u}{\tau }%
)\leq 1\}.%
\end{array}%
\end{equation*}%
Recall for any $u\in L^{p(x)}(\Omega )$ it holds%
\begin{equation}
\left\{ 
\begin{array}{l}
\left\Vert u\right\Vert _{p(x)}^{p^{-}}\leq \rho _{p(x)}(u)\leq \left\Vert
u\right\Vert _{p(x)}^{p^{+}}\text{ \ if \ }\left\Vert u\right\Vert _{p(x)}>1,
\\ 
\left\Vert u\right\Vert _{p(x)}^{p^{+}}\leq \rho _{p(x)}(u)\leq \left\Vert
u\right\Vert _{p(x)}^{p^{-}}\text{ \ if \ }\left\Vert u\right\Vert
_{p(x)}\leq 1,%
\end{array}%
\right.  \label{10}
\end{equation}%
and 
\begin{equation}
\left\Vert u\right\Vert _{p(x)}=a\text{ \ if and only if }\rho _{p(x)}(\frac{%
u}{a})=1.  \label{normro}
\end{equation}%
The variable exponent Sobolev space $W^{1,p(\cdot )}(\Omega ),$ defined by%
\begin{equation*}
\begin{array}{l}
W^{1,p(x)}(\Omega )=\{u\in L^{p(x)}(\Omega ):|\nabla u|\in L^{p(x)}(\Omega
)\},%
\end{array}%
\end{equation*}%
is endowed with the norm $\left\Vert u\right\Vert _{1,p(x)}=\left\Vert
\nabla u\right\Vert _{p(x)}$ which makes it a Banach space. In the sequel, $%
d(x):=d(x,\partial \Omega )$ denotes the euclidean distance of $x$ with
respect to the boundary $\partial \Omega $.

The next lemma is a slight modification of \cite[Lemma 5.1]{NV} which will
be useful later on.

\begin{lemma}
\label{NV_new} Let $k,m\in L^{\infty }(\Omega )$ be two real and positive
functions with $m^{-}>0$. If $u\in L^{k(x)}\left( \Omega \right) $ then $%
u^{m(x)}\in L^{\frac{k(x)}{m(x)}}\left( \Omega \right) $, and there exists $%
x_{0}\in \Omega $ such that 
\begin{equation*}
\begin{array}{l}
\left\Vert |u|^{m(x)}\right\Vert _{\frac{k(x)}{m(x)}}=\Vert u\Vert
_{k(x)}^{m(x_{0})}.%
\end{array}%
\end{equation*}
\end{lemma}

\begin{proof}
On account of (\ref{normro}), Mean Value theorem \cite[Theorem 5]{B} ensures
the existence of $x_{0}\in \Omega $ such that 
\begin{equation*}
\begin{array}{l}
1=\rho _{\frac{k(x)}{m(x)}}(\frac{|u|^{m(x)}}{\left\Vert
|u|^{m(x)}\right\Vert _{\frac{k(x)}{m(x)}}})=\int_{\Omega }|\frac{u}{\Vert
u\Vert _{k(x)}}|^{k(x)}\frac{\Vert u\Vert _{k(x)}^{k(x)}}{\left\Vert
|u_{1}|^{m(x)}\right\Vert _{\frac{k(x)}{m(x)}}^{\frac{k(x)}{m(x)}}}dx \\ 
=\frac{\Vert u\Vert _{k(x)}^{k(x_{0})}}{\left\Vert |u_{1}|^{m(x)}\right\Vert
_{\frac{k(x)}{m(x)}}^{\frac{k(x_{0})}{m(x_{0})}}}\rho _{k(x)}(\frac{u}{\Vert
u\Vert _{k(x)}})=\frac{\Vert u\Vert _{k(x)}^{k(x_{0})}}{\left\Vert
|u|^{m(x)}\right\Vert _{\frac{k(x)}{m(x)}}^{\frac{k(x_{0})}{m(x_{0})}}},%
\end{array}%
\end{equation*}%
showing the desired identity.
\end{proof}

A priori gradient estimate is provided in the next lemma. It is a partial
extension of \cite[Lemma $1$]{BE} to problems involving variable exponents.

\begin{lemma}
\label{L2}Let $h\in L^{\infty }(\Omega )$ be a nontrivial sign-constant
function and let $u\in W_{0}^{1,p(x)}(\Omega )$ be the weak solution of the
Dirichlet problem 
\begin{equation}
-\Delta _{p(x)}u=h(x)\text{ in }\Omega ,\text{ }u=0\text{ on }\partial
\Omega .  \label{1}
\end{equation}%
Then, there exists a constant $\bar{k}_{p}>0$, depending only on $p$, $N$,
and $\Omega $, such that 
\begin{equation}
\Vert \nabla u\Vert _{\infty }\leq \bar{k}_{p}\Vert h\Vert _{\infty }^{\frac{%
1}{p^{\pm }-1}}  \label{2}
\end{equation}%
with%
\begin{equation*}
p^{\pm }:=\left\{ 
\begin{array}{ll}
p^{-} & \text{if }\Vert h\Vert _{\infty }>1 \\ 
p^{+} & \text{if }\Vert h\Vert _{\infty }\leq 1.%
\end{array}%
\right. 
\end{equation*}
\end{lemma}

\begin{proof}
First, assume that $\Vert h\Vert _{\infty }\leq 1.$ Multiplying problem (\ref%
{1}) by $\varphi \in W_{0}^{1,p(.)}(\Omega ),$ with $\varphi \geq 0,$
integrating over $\Omega $ we obtain 
\begin{equation*}
\int_{\Omega }|\nabla u|^{p(x)-2}\nabla u\nabla \varphi \text{ }%
dx=\int_{\Omega }h(x)\varphi \text{ }dx\leq \int_{\Omega }\varphi \text{ }%
dx=\int_{\Omega }|\nabla \xi |^{p(x)-2}\nabla \xi \nabla \varphi \text{ }dx,
\end{equation*}%
where $\xi (x)$ is the $p(x)$-torsion function defined by 
\begin{equation*}
-\Delta _{p(x)}\xi =1\text{ in }\Omega ,\text{ \ }\xi =0\text{ on }\partial
\Omega .
\end{equation*}%
The weak comparison principle implies $\Vert u\Vert _{\infty }\leq \Vert \xi
\Vert _{\infty }$ while the regularity theorem in \cite{F} ensures the
existence of constants $\tau \in (0,1)$ and $\bar{k}_{p}>0$ such that $\Vert
u\Vert _{C^{1,\tau }(\overline{\Omega })}\leq \bar{k}_{p}.$

Now we deal with the case $\Vert h\Vert _{\infty }>1$. By Theorem \ref{MVT2} in the Appendix, there exists $x_{0}\in
\Omega $ such that%
\begin{equation*}
\begin{array}{l}
\int_{\Omega }|\nabla (\Vert h\Vert _{\infty }^{\frac{-1}{p^{-}-1}%
}u)|^{p(x)-2}\nabla (\Vert h\Vert _{\infty }^{\frac{-1}{p^{-}-1}}u)\nabla
\varphi \text{ }dx \\ 
=\int_{\Omega }\Vert h\Vert _{\infty }^{\frac{-(p(x)-1)}{p^{-}-1}}|\nabla
u|^{p(x)-2}\nabla u\nabla \varphi \text{ }dx \\ 
=\Vert h\Vert _{\infty }^{\frac{-(p(x_{0})-1)}{p^{-}-1}}\int_{\Omega
}|\nabla u|^{p(x)-2}\nabla u\nabla \varphi \text{ }dx=\Vert h\Vert _{\infty
}^{\frac{-(p(x_{0})-1)}{p^{-}-1}}\int_{\Omega }h(x)\varphi \text{ }dx \\ 
\leq \Vert h\Vert _{\infty }^{\frac{-(p^{-}-1)}{p^{-}-1}}\int_{\Omega
}h(x)\varphi \text{ }dx=\int_{\Omega }\Vert h\Vert _{\infty
}^{-1}h(x)\varphi \text{ }dx\leq \int_{\Omega }\varphi \text{ }dx.%
\end{array}%
\end{equation*}%
Thus, in view of the previous argument, it follows that 
\begin{equation*}
\Vert h\Vert _{\infty }^{\frac{-1}{p^{-}-1}}|u|_{1,\alpha }\leq \bar{k}_{p},
\end{equation*}%
showing that (\ref{2}) holds true. This ends the proof.
\end{proof}

The case when $h$ in (\ref{1}) is not an $L^{\infty }$-bounded function is
handled in the next lemma which provides an a priori $L^{\infty }$-estimate
of solutions for (\ref{1}).

\begin{lemma}
\label{L4}Assume $h\in L^{p^{\prime }(x)}(\Omega )\cap L^{N}(\Omega )$ in (%
\ref{1}). Then, there exists a constant $C>0$, depending only on $N,p$ and $%
\Omega $, such that%
\begin{equation}
\begin{array}{l}
\Vert u\Vert _{\infty }\leq C\left\Vert h\right\Vert _{L^{N}(\Omega )}^{%
\frac{1}{p^{\pm }-1}},%
\end{array}
\label{u}
\end{equation}%
with%
\begin{equation*}
p^{\pm }:=\left\{ 
\begin{array}{ll}
p^{-} & \text{if }\left\Vert h\right\Vert _{L^{N}(\Omega )}>1 \\ 
p^{+} & \text{if }\left\Vert h\right\Vert _{L^{N}(\Omega )}\leq 1.%
\end{array}%
\right.
\end{equation*}
\end{lemma}

\begin{proof}
For each $k\in 
\mathbb{N}
,$ consider the set $A_{k}=\{x\in \Omega :u(x)>k\}$, where $u$ is the
solution of (\ref{1}). Thus, to prove (\ref{u}) amounts to show that%
\begin{equation*}
|A_{k}|=0\,\,\mbox{ for any }\,\,k>k_{0},
\end{equation*}%
where 
\begin{equation}
k_{0}:=C\left\Vert h\right\Vert _{L^{N}(\Omega )}^{\frac{1}{p^{\pm }-1}},
\label{C_omega}
\end{equation}%
with a constant $C>0$ that will be chosen later on. By contradiction assume
that there exists $k>k_{0}$ such that $|A_{k}|\neq 0$. Testing (\ref{1})
with $(u-k)^{+}$ leads to 
\begin{equation}
\int_{A_{k}}\left\vert \nabla u\right\vert ^{p(x)}dx=\int_{A_{k}}h\left(
x\right) (u-k)^{+}dx\leq \left\Vert h\right\Vert _{L^{N}(\Omega )}\left\Vert
(u-k)^{+}\right\Vert _{L^{N^{\prime }}(A_{k})}.  \label{a_k_1}
\end{equation}%
One has 
\begin{equation}
\begin{array}{l}
\int_{A_{k}}\left\vert \nabla u\right\vert ^{p(x)}dx\geq \Vert \nabla u\Vert
_{L^{p^{-}}(A_{k}^{-})}^{p^{-}}+\Vert \nabla u\Vert
_{L^{p^{+}}(A_{k}^{+})}^{p^{+}},%
\end{array}%
\end{equation}%
where $A_{k}^{-}:=A_{k}\cap \{|\nabla u|\geq 1\}$ and $A_{k}^{+}:=A_{k}\cap
\{|\nabla u|<1\}$. H\"{o}lder's inequality together with classical Sobolev
and Lebesgue embeddings imply 
\begin{equation}
\begin{array}{l}
\left\Vert (u-k)^{+}\right\Vert _{L^{N^{\prime }}(A_{k})}\leq
C_{1}\int_{A_{k}}|\nabla u|dx \\ 
\leq C_{1}\left( |A_{k}^{-}|^{\frac{p^{-}-1}{p^{-}}}\Vert \nabla u\Vert
_{L^{p^{-}}(A_{k}^{-})}+|A_{k}^{+}|^{\frac{p^{+}-1}{p^{+}}}\Vert \nabla
u\Vert _{L^{p^{+}}(A_{k}^{+})}\right) ,%
\end{array}
\label{a_k_2}
\end{equation}%
as well as%
\begin{equation}
\begin{array}{l}
\Vert \nabla u\Vert _{L^{p^{\pm }}(A_{k}^{\pm })}^{p^{\pm }}\geq \Vert
\nabla u\Vert _{L^{p^{\pm }}(A_{k}^{\pm })}(\Vert \nabla u\Vert
_{L^{1}(A_{k}^{\pm })}|A_{k}^{\pm }|^{-\frac{p^{\pm }-1}{p^{\pm }}})^{p^{\pm
}-1} \\ 
\geq \Vert \nabla u\Vert _{L^{p^{\pm }}(A_{k}^{\pm })}(C_{1}^{-1}\Vert
u\Vert _{L^{N^{\prime }}(A_{k}^{\pm })}|A_{k}^{\pm }|^{-\frac{p^{\pm }-1}{%
p^{\pm }}})^{p^{\pm }-1} \\ 
\geq \Vert \nabla u\Vert _{L^{p^{\pm }}(A_{k}^{\pm
})}(C_{1}^{-1}k|A_{k}^{\pm }|^{\frac{1}{N^{\prime }}-\frac{p^{\pm }-1}{%
p^{\pm }}})^{p^{\pm }-1},%
\end{array}
\label{ak3}
\end{equation}%
for a certain constant $C_{1}>0$. Then, gathering (\ref{a_k_1})-(\ref{ak3})
together, we infer that 
\begin{equation}
\begin{array}{l}
\Vert \nabla u\Vert _{L^{p^{-}}(A_{k}^{-})}\left[ (C_{1}^{-1}k|A_{k}^{-}|^{%
\frac{1}{N^{\prime }}-\frac{p^{-}-1}{p^{-}}})^{p^{-}-1}-C_{1}\left\Vert
h\right\Vert _{L^{N}(\Omega )}|A_{k}^{-}|^{\frac{p^{-}-1}{p^{-}}}\right] \\ 
+\Vert \nabla u\Vert _{L^{p^{+}}(A_{k}^{+})}\left[ (C_{1}^{-1}k|A_{k}^{+}|^{%
\frac{1}{N^{\prime }}-\frac{p^{+}-1}{p^{+}}})^{p^{+}-1}-C_{1}\left\Vert
h\right\Vert _{L^{N}(\Omega )}|A_{k}^{+}|^{\frac{p^{+}-1}{p^{+}}}\right] \\ 
\leq 0.%
\end{array}
\label{combinaison}
\end{equation}%
Now, fix $C>0$ in (\ref{C_omega}) as follows \ 
\begin{equation*}
C:=\max \{C_{1}^{p^{-\prime }},C_{1}^{p^{+\prime }}\}|\Omega |^{1-\frac{1}{%
N^{\prime }}}.
\end{equation*}%
Then, for $k>k_{0},$ it follows that%
\begin{equation*}
\left[ (C_{1}^{-1}k|A_{k}^{\pm }|^{\frac{1}{N^{\prime }}-\frac{p^{\pm }-1}{%
p^{\pm }}})^{p^{\pm }-1}-C_{1}\left\Vert h\right\Vert _{L^{N}(\Omega
)}|A_{k}^{\pm }|^{\frac{p^{\pm }-1}{p^{\pm }}}\right] \geq 0,
\end{equation*}%
which contradicts (\ref{combinaison}). This completes the proof.
\end{proof}

\section{An auxiliary system}

\label{S3}

For each $(z_{1},z_{2})\in W_{0}^{1,p_{1}(x)}(\Omega )\times
W_{0}^{1,p_{2}(x)}(\Omega ),$ we consider the auxiliary problem 
\begin{equation*}
(\mathrm{P}_{(z_{1},z_{2})})\qquad \left\{ 
\begin{array}{l}
-\Delta _{p_{i}(x)}u_{i}=f_{i}(x,z_{1},z_{2},\nabla z_{1},\nabla z_{2})\text{
\ in }\Omega \\ 
u_{i}=0\text{ \ on }\partial \Omega ,\text{ }i=1,2.%
\end{array}%
\right.
\end{equation*}

\begin{lemma}
\label{weak_bound} Assume $($\textrm{\~{H}}$_{\alpha ,\beta ,\gamma })$
holds. Suppose that 
\begin{equation}
z_{i}(x)\geq \tilde{c}d(x)\text{ \ and \ }\Vert \nabla z_{i}\Vert
_{p_{i}(x)}\leq \tilde{L},\,  \label{99}
\end{equation}%
for some constants $\tilde{c},\tilde{L}>0$ independent on $z_{1}$ and $z_{2}$%
. Then, for $\tilde{L}$ large enough, problem $(\mathrm{P}_{(z_{1},z_{2})})$
admits a unique solution $(u_{1},u_{2})$ in $W_{0}^{1,p_{1}(x)}(\Omega
)\times W_{0}^{1,p_{2}(x)}(\Omega )$ satisfying 
\begin{equation}
\Vert \nabla u_{i}\Vert _{p_{i}(x)}\leq \tilde{L},\,\,i=1,2.
\label{estimation_z}
\end{equation}
\end{lemma}

\begin{proof}
First, we claim that 
\begin{equation}
f_{i}(.,z_{1},z_{2},\nabla z_{1},\nabla z_{2})\in L^{p_{i}^{\prime }(x)
}(\Omega )\cap L^{N}(\Omega ),\,\text{for }i=1,2.  \label{belonging}
\end{equation}%
We only show that $f_{i}\in L^{p_{i}^{\prime }(x) }(\Omega )$ in (\ref%
{belonging}) because $f_{i}\in L^{N}(\Omega )$ can be justified similarly by
substituting $p_{i}^{\prime }(.)$ with $N$ in the argument below. By $($%
\textrm{H}$_{f})$ we have%
\begin{equation}
\begin{array}{l}
\left\Vert f_{i}\left(.,z_{1},z_{2},\nabla z_{1},\nabla z_{2}\right)
\right\Vert _{p_{i}^{\prime }(x) } \\ 
\leq M_{i}\left( ||z_{1}^{\alpha _{i}(.)}z_{2}^{\beta
_{i}(.)}||_{p_{i}^{\prime }(x) }+\left\Vert |\nabla z_{1}|^{\gamma
_{i}(.)}\right\Vert _{p_{i}^{\prime }(x) }+\left\Vert |\nabla z_{2}|^{%
\overline{\gamma }_{i}(.)}\right\Vert _{p_{i}^{\prime }(x) }\right) .%
\end{array}
\label{f}
\end{equation}%
Lemma \ref{NV_new} together with $($\textrm{\~{H}}$_{\alpha ,\beta ,\gamma
}) $ imply 
\begin{equation}
\begin{array}{l}
\left\Vert |\nabla z_{1}|^{\gamma _{i}(.)}\right\Vert _{p_{i}^{\prime }(x)
}+\left\Vert |\nabla z_{2}|^{\overline{\gamma }_{i}(.)}\right\Vert
_{p_{i}^{\prime }(x) }=\left\Vert \nabla z_{1}\right\Vert _{\gamma _{i}(x)
p_{i}^{\prime }(x) }^{\gamma _{i}\left( x_{0}^{i}\right) }+\left\Vert \nabla
z_{2}\right\Vert _{\overline{\gamma }_{i}(x) p_{i}^{\prime }(x) }^{\overline{%
\gamma }_{i}\left( x_{1}^{i}\right) } \\ 
\leq C(\left\Vert \nabla z_{1}\right\Vert _{p_{1}(x)}^{\gamma _{i}\left(
x_{0}^{i}\right) }+\left\Vert \nabla z_{2}\right\Vert _{p_{2}(x)}^{\overline{%
\gamma }_{i}\left( x_{1}^{i}\right) }),%
\end{array}
\label{j,g}
\end{equation}%
for certain $x_{0}^{i},x_{1}^{i}\in \Omega $ and a constant $C>0$. From (\ref%
{99}) we have%
\begin{equation*}
\begin{array}{l}
||z_{1}^{\alpha _{i}(.)}z_{2}^{\beta _{i}(.)}||_{p_{i}^{\prime }(x)}\leq
\left\{ 
\begin{array}{ll}
||(\tilde{c}d(.))^{\alpha _{i}(.)+\beta _{i}(.)}||_{p_{i}^{\prime }(x)} & 
\text{if }\alpha _{i}^{+},\beta _{i}^{+}<0 \\ 
||(\tilde{c}d(.))^{\alpha _{i}(.)}z_{2}^{\beta _{i}(.)}||_{p_{i}^{\prime
}(x)} & \text{if }\alpha _{i}^{+}<0<\beta _{i}^{-} \\ 
||(\tilde{c}d(.))^{\beta _{i}(.)}z_{1}^{\alpha _{i}(.)}||_{p_{i}^{\prime
}(x)} & \text{if }\beta _{i}^{+}<0<\alpha _{i}^{-}.%
\end{array}%
\right.%
\end{array}%
\end{equation*}%
By $($\textrm{\~{H}}$_{\alpha ,\beta ,\gamma })$, H\"{o}lder's inequality
gives%
\begin{equation}
\begin{array}{l}
||z_{1}^{\alpha _{i}(.)}z_{2}^{\beta _{i}(.)}||_{p_{i}^{\prime }(x)}\leq
C_{0}\left\{ 
\begin{array}{ll}
||(\tilde{c}d(.))^{\alpha _{i}(.)+\beta _{i}(.)}||_{p_{i}^{\prime }(x)} & 
\text{if }\alpha _{i}^{+},\beta _{i}^{+}<0 \\ 
\left\Vert (\tilde{c}d(.))^{\alpha _{i}(.)}\right\Vert _{\frac{N^{\prime
}p_{i}^{\prime }(x) }{N^{\prime }-p_{i}^{\prime }(x) \beta _{i}(x) }%
}||z_{2}^{\beta _{i}(.)}||_{\frac{N^{\prime }}{\beta _{i}(x) }} & \text{if }%
\alpha _{i}^{+}<0<\beta _{i}^{-} \\ 
\left\Vert (\tilde{c}d(.))^{\beta _{i}(.)}\right\Vert _{\frac{N^{\prime
}p_{i}^{\prime }(x) }{N^{\prime }-p_{i}^{\prime }(x) \alpha _{i}(x) }%
}||z_{1}^{\alpha _{i}(.)}||_{\frac{N^{\prime }}{\alpha _{i}(x) }} & \text{if 
}\beta _{i}^{+}<0<\alpha _{i}^{-}.%
\end{array}%
\right.%
\end{array}
\label{19}
\end{equation}%
Observe that%
\begin{equation*}
\begin{array}{l}
\int_{\Omega }d(x)^{(\alpha _{i}(x) +\beta _{i}(x) )p_{i}^{\prime }(x) }%
\text{ }dx \\ 
=\int_{\{d\geq 1\}}d(x)^{(\alpha _{i}(x) +\beta _{i}(x) )p_{i}^{\prime }(x) }%
\text{ }dx+\int_{\{d<1\}}d(x)^{(\alpha _{i}(x) +\beta _{i}(x) )p_{i}^{\prime
}(x) }\text{ }dx \\ 
\leq |\Omega |+\int_{\{d<1\}}d(x)^{(\alpha _{i}^{+}+\beta
_{i}^{+})(p_{i}^{\prime })^{+}}\text{ }dx.%
\end{array}%
\end{equation*}%
Then, owing to \cite[Lemma in page 726]{LM}, which is applicable since $%
(\alpha _{i}^{+}+\beta _{i}^{+})(p_{i}^{\prime })^{+}>-1$ (see $($\textrm{\~{%
H}}$_{\alpha ,\beta ,\gamma })$), we infer that%
\begin{equation}
\begin{array}{l}
\int_{\Omega }d(x)^{(\alpha _{i}(x) +\beta _{i}(x) )p_{i}^{\prime }(x) }%
\text{ }dx<+\infty .%
\end{array}
\label{lazer-mackenna}
\end{equation}%
Thence, from (\ref{10}), we derive that%
\begin{equation}
\begin{array}{l}
||(\tilde{c}d(.))^{\alpha _{i}(.)+\beta _{i}(.)}||_{p_{i}^{\prime }(x)
}<+\infty .%
\end{array}
\label{31}
\end{equation}%
On account of $($\textrm{\~{H}}$_{\alpha ,\beta ,\gamma })$ the same
conclusion can be drawn for the cases $\alpha _{i}^{+}<0<\beta _{i}^{-}$ and 
$\beta _{i}^{+}<0<\alpha _{i}^{-}.$ Hence, a similar argument as above
produces%
\begin{equation}
\begin{array}{l}
\left\Vert (\tilde{c}d(.))^{\alpha _{i}(.)}\right\Vert _{\frac{N^{\prime
}p_{i}^{\prime }(x) }{N^{\prime }-p_{i}^{\prime }(x) \beta _{i}(x) }},\text{
\ }\left\Vert (\tilde{c}d(.))^{\beta _{i}(.)}\right\Vert _{\frac{N^{\prime
}p_{i}^{\prime }(x) }{N^{\prime }-p_{i}^{\prime }(x) \alpha _{i}(x) }%
}<+\infty .%
\end{array}
\label{31*}
\end{equation}%
Reporting (\ref{31})-(\ref{31*}) in (\ref{19}), by Lemma \ref{NV_new}, there
exist $x_{2}^{i},x_{3}^{i}\in \Omega $ such that%
\begin{equation*}
\begin{array}{l}
||z_{1}^{\alpha _{i}(.)}z_{2}^{\beta _{i}(.)}||_{p_{i}^{\prime }(x) }\leq
C_{1}(1+||z_{1}^{|\alpha _{i}(.)|}||_{\frac{N^{\prime }}{|\alpha _{i}(x) |}%
}+||z_{2}^{|\beta _{i}(.)|}||_{\frac{N^{\prime }}{|\beta _{i}(x) |}}) \\ 
=C_{1}(1+||z_{1}||_{N^{\prime }}^{|\alpha
_{i}(x_{2}^{i})|}+||z_{2}||_{N^{\prime }}^{|\beta _{i}(x_{3}^{i})|}),%
\end{array}%
\end{equation*}%
where $C_{1}>0$ is a constant. Sobolev embedding $W_{0}^{1,p_{i}(x)}(\Omega
)\hookrightarrow L^{N^{\prime }}(\Omega )$ together with H\"{o}lder's
inequality lead to%
\begin{equation}
\begin{array}{l}
||z_{1}^{\alpha _{i}(.)}z_{2}^{\beta _{i}(.)}||_{p_{i}^{\prime }(x) }\leq 
\tilde{C}_{1}(1+\left\Vert \nabla z_{1}\right\Vert _{p_{1}(x)}^{|\alpha
_{i}(x_{2}^{i})|}+\left\Vert \nabla z_{2}\right\Vert _{p_{2}(x)}^{|\beta
_{i}(x_{3}^{i})|}),%
\end{array}
\label{ab}
\end{equation}%
for some constant $\tilde{C}_{1}>0$. Gathering (\ref{f}), (\ref{j,g}) and (%
\ref{ab}) together it follows that%
\begin{equation}
\begin{array}{l}
\left\Vert f_{i}\left(.,z_{1},z_{2},\nabla z_{1},\nabla z_{2}\right)
\right\Vert _{p_{i}^{\prime }(x) } \\ 
\leq C_{2}(1+\left\Vert \nabla z_{1}\right\Vert _{p_{1}(x)}^{|\alpha
_{i}(x_{2}^{i})|}+\left\Vert \nabla z_{2}\right\Vert _{p_{2}(x)}^{|\beta
_{i}(x_{3}^{i})|}\left\Vert \nabla z_{1}\right\Vert _{p_{1}(x)}^{\gamma
_{i}\left( x_{0}^{i}\right) }+\left\Vert \nabla z_{2}\right\Vert
_{p_{2}(x)}^{\overline{\gamma }_{i}\left( x_{1}^{i}\right) }),%
\end{array}
\label{21}
\end{equation}%
for certain constant $C_{2}>0$. Repeating the argument above by starting in (%
\ref{f}) with $N$ instead of $p_{i}^{\prime }(.)$ and by using $($\textrm{\~{%
H}}$_{\alpha ,\beta ,\gamma })$, we get%
\begin{equation}
\begin{array}{l}
\left\Vert f_{i}\left(.,z_{1},z_{2},\nabla z_{1},\nabla z_{2}\right)
\right\Vert _{N} \\ 
\leq \tilde{C}_{2}(1+\left\Vert \nabla z_{1}\right\Vert _{p_{1}(x)}^{|\alpha
_{i}(\hat{x}_{2}^{i})|}+\left\Vert \nabla z_{2}\right\Vert
_{p_{2}(x)}^{|\beta _{i}(\hat{x}_{3}^{i})|}+\left\Vert \nabla
z_{1}\right\Vert _{p_{1}(x)}^{\gamma _{i}(\hat{x}_{0}^{i})}+\left\Vert
\nabla z_{2}\right\Vert _{p_{2}(x)}^{\overline{\gamma }_{i}(\hat{x}%
_{1}^{i})}) \\ 
\leq \tilde{C}_{2}(1+\left\Vert \nabla z_{1}\right\Vert _{p_{1}(x)}^{\max
\{|\alpha _{i}(\hat{x}_{2}^{i})|,\gamma _{i}(\hat{x}_{0}^{i})\}}+\left\Vert
\nabla z_{2}\right\Vert _{p_{2}(x)}^{\max \{|\beta _{i}(\hat{x}_{3}^{i})|,%
\overline{\gamma }_{i}(\hat{x}_{1}^{i})\}}),%
\end{array}
\label{21*}
\end{equation}%
where $\tilde{C}_{2}>0$ is a constant. Hence, on the basis of (\ref{99}),
the claim follows.

Consequently, the unique solvability of $(\mathrm{P}_{(z_{1},z_{2})})$ comes
directly from Browder-Minty Theorem (see, e.g., \cite{BR}).

The task is now to show that the estimate (\ref{estimation_z}) holds true.
Thanks to Lemma \ref{NV_new}, there exist $x_{4}^{i}\in \Omega $ such that%
\begin{equation}
\begin{array}{l}
\Vert \nabla u_{i}\Vert _{p_{i}(x)}^{p_{i}\left( x_{4}^{i}\right)
}=\int_{\Omega }|\nabla u_{i}|^{p_{i}(x)}dx.%
\end{array}
\label{22}
\end{equation}%
Testing $(\mathrm{P}_{(z_{1},z_{2})})$ with $(u_{1},u_{2})$, H\"{o}lder's
inequality and the embedding $W_{0}^{1,p_{i}(x)}(\Omega )\hookrightarrow
L^{N^{\prime }}(\Omega )$ entail%
\begin{equation}
\begin{array}{l}
\int_{\Omega }|\nabla u_{i}|^{p_{i}(x)}dx\leq \int_{\Omega }f_{i}\left(
x,z_{1},z_{2},\nabla z_{1},\nabla z_{2}\right) u_{1}dx \\ 
\leq C_{0}\left\Vert f_{i}\left(.,z_{1},z_{2},\nabla z_{1},\nabla
z_{2}\right) \right\Vert _{N}\Vert \nabla u_{i}\Vert _{p_{i}(x)},%
\end{array}
\label{23}
\end{equation}%
for a constant $C_{0}>0$. Combining (\ref{22})-(\ref{23}) with (\ref{21*})
and bearing in mind (\ref{99}), one derives that 
\begin{equation*}
\begin{array}{l}
\Vert \nabla u_{i}\Vert _{p_{i}(x)}\leq \lbrack C_{0}\left\Vert f_{i}\left(
.,z_{1},z_{2},\nabla z_{1},\nabla z_{2}\right) \right\Vert _{N}]^{\frac{1}{%
p_{i}(x_{4}^{i})-1}} \\ 
\leq \tilde{C}_{3}(1+\left\Vert \nabla z_{1}\right\Vert _{p_{1}(x)}^{\max
\{|\alpha _{i}(\hat{x}_{2}^{i})|,\gamma _{i}(\hat{x}_{0}^{i})\}}+\left\Vert
\nabla z_{2}\right\Vert _{p_{2}(x)}^{\max \{|\beta _{i}(\hat{x}_{3}^{i})|,%
\overline{\gamma }_{i}(\hat{x}_{1}^{i})\}})^{\frac{1}{p_{i}(x_{4}^{i})-1}}
\\ 
\leq \tilde{C}_{3}(1+\tilde{L}^{\max \{|\alpha _{i}(\hat{x}_{2}^{i})|,\gamma
_{i}(\hat{x}_{0}^{i})\}}+\tilde{L}^{\max \{|\beta _{i}(\hat{x}_{3}^{i})|,%
\overline{\gamma }_{i}(\hat{x}_{1}^{i})\}})^{\frac{1}{p_{i}(x_{4}^{i})-1}%
}\leq \tilde{L},%
\end{array}%
\end{equation*}%
provided that $\tilde{L}>0$ is sufficiently large, where $\tilde{C}_{3}>0$
is a constant independent of $z_{i}$. This is possible because, according to 
$($\textrm{\~{H}}$_{\alpha ,\beta ,\gamma })$, one has%
\begin{equation*}
\begin{array}{l}
\max \{|\alpha _{i}(\hat{x}_{2}^{i})|,|\beta _{i}(\hat{x}_{3}^{i})|,\gamma
_{i}(\hat{x}_{0}^{i}),\overline{\gamma }_{i}(\hat{x}_{1}^{i})%
\}<p_{i}(x_{4}^{i})-1.%
\end{array}%
\end{equation*}%
This completes the proof.
\end{proof}

\begin{lemma}
\label{L3}Under assumptions \textrm{(H}$_{f}$\textrm{)} and $($\textrm{\~{H}}%
$_{\alpha ,\beta ,\gamma })$, for $(z_{1},z_{2})$ satisfying (\ref{99}),
there exists a constant $L>1$ independent of $z_{i}$ such that every
solution $(u_{1},u_{2})$ of $(\mathrm{P}_{(z_{1},z_{2})})$ belongs to $%
L^{\infty }(\Omega )\times L^{\infty }(\Omega )$ and satisfies the estimate%
\begin{equation}
\left\Vert u_{i}\right\Vert _{\infty }<L.  \label{6}
\end{equation}
\end{lemma}

\begin{proof}
It is a direct consequence of Lemma \ref{L4} where (\ref{21*}) as well as $($%
\textrm{\~{H}}$_{\alpha ,\beta ,\gamma })$ and (\ref{99}) are used.
\end{proof}

\section{Comparison results}

\label{S4}

Let $\xi _{i},\xi _{i,\delta }\in C^{1,\tau }(\overline{\Omega }),$ $\tau
\in (0,1)$, be the solutions of the Dirichlet problems 
\begin{equation}
-\Delta _{p_{i}(x)}\xi _{i}(x)=1\text{\ in }\Omega ,\text{ \ }\xi _{i}(x)=0%
\text{ \ on }\partial \Omega  \label{9*}
\end{equation}%
and%
\begin{equation}
-\Delta _{p_{i}}\xi _{i,\delta }(x)=\left\{ 
\begin{array}{ll}
1 & \text{in }\Omega \backslash \overline{\Omega }_{\delta } \\ 
-1 & \text{in }\Omega _{\delta }%
\end{array}%
\right. ,\text{ }\xi _{i,\delta }(x)=0\text{ \ on }\partial \Omega ,
\label{9}
\end{equation}%
where 
\begin{equation*}
\Omega _{\delta }:=\{x\in \Omega :d(x)<\delta \},
\end{equation*}%
with a fixed $\delta >0$ sufficiently small.

\begin{lemma}
\label{L5}There are constants $\tau >0$ and $%
c_{1},k_{p_{1}},k_{p_{2}}>1>c_{0}$ such that%
\begin{equation}
c_{0}d(x)\leq \xi _{i,\delta }(x)\leq \xi _{i}(x)\leq c_{1}d(x)\text{ \ for
all }x\in \Omega  \label{5}
\end{equation}%
and%
\begin{equation}
\left\Vert \xi _{i,\delta }\right\Vert _{C^{1,\tau }(\overline{\Omega }%
)},\left\Vert \xi _{i}\right\Vert _{C^{1,\tau }(\overline{\Omega })}\leq
k_{p_{i}}\text{, }i=1,2.  \label{5**}
\end{equation}
\end{lemma}

\begin{proof}
From (\ref{9*}) and (\ref{9}), it is readily seen that $\xi _{i,\delta
}(x)\leq \xi _{i}(x)$ for all $x\in \Omega ,$ for $i=1,2.$ The Strong
Maximum Principle together with \cite[Lemma 3]{AM} entail $\xi _{i,\delta
}(x)\geq c_{0}d(x)$ in $\Omega ,$ for $\delta >0$ sufficiently small in (\ref%
{9}) while, invoking Lemma \ref{L2}, we infer that (\ref{5**}) holds true.
Moreover, using (\ref{5**}), a similar argument to that in the proof of \cite%
[Lemma 3.1]{DM} shows that the last inequality in (\ref{5}) is verified.
This ends the proof.
\end{proof}

For a constant $C>1$ set%
\begin{equation}
\underline{u}_{i}=C^{-1}\xi _{i,\delta }\quad \text{and}\quad \overline{u}%
_{i}=C\xi _{i}.  \label{sub-super}
\end{equation}

We claim that $\overline{u}_{i}\geq \underline{u}_{i}$ in $\overline{\Omega }
$. Indeed, observe, from (\ref{9*}) and (\ref{9}), that the integrals
\begin{equation}
\int_{\Omega \backslash \overline{\Omega }_{\delta }}|\nabla \xi _{i,\delta
}|^{p_{i}(x)-2}\nabla \xi _{i,\delta }\nabla \varphi _{i}\text{ }dx,
\label{34**}
\end{equation}%
\begin{equation}
-\int_{\Omega _{\delta }}|\nabla \xi _{i,\delta }|^{p_{i}(x)-2}\nabla \xi
_{i,\delta }\nabla \varphi _{i}\text{ }dx,  \label{34*}
\end{equation}%
\begin{equation}
\int_{\Omega }|\nabla \xi _{i}|^{p_{i}(x)-2}\nabla \xi _{i}\nabla \varphi
_{i}\text{ }dx,  \label{34}
\end{equation}%
are positive for all $\varphi _{i}\in W_{0}^{1,p_{i}(x)}(\Omega )$ with $%
\varphi _{i}\geq 0$. This is crucial so that Theorem \ref{MVT2} in the
Appendix is applicable. By (\ref{sub-super}) and thanks to Theorem \ref{MVT2}%
, there exist $x_{i}^{1},x_{i}^{2}\in \Omega $ such that%
\begin{equation}
\begin{array}{l}
\int_{\Omega }|\nabla \underline{u}_{i}|^{p_{i}(x)-2}\nabla \underline{u}%
_{i}\nabla \varphi _{i}\text{ }dx \\ 
=\int_{\Omega }C^{-(p_{i}(x)-1)}|\nabla \xi _{i,\delta }|^{p_{i}(x)-2}\nabla
\xi _{i,\delta }\nabla \varphi _{i}\text{ }dx \\ 
=\int_{\Omega \backslash \overline{\Omega }_{\delta
}}C^{-(p_{i}(x)-1)}|\nabla \xi _{i,\delta }|^{p_{i}(x)-2}\nabla \xi
_{i,\delta }\nabla \varphi _{i}\text{ }dx \\ 
\text{ \ \ \ \ \ }-\int_{\Omega _{\delta }}(-C^{-(p_{i}(x)-1)}|\nabla \xi
_{i,\delta }|^{p_{i}(x)-2}\nabla \xi _{i,\delta }\nabla \varphi _{i})\text{ }%
dx \\ 
=C^{-(p_{i}(x_{i}^{1})-1)}\int_{\Omega \backslash \overline{\Omega }_{\delta
}}|\nabla \xi _{i,\delta }|^{p_{i}(x)-2}\nabla \xi _{i,\delta }\nabla
\varphi _{i}\text{ }dx \\ 
\text{ \ \ \ \ \ }-C^{-(p_{i}(x_{i}^{2})-1)}\int_{\Omega _{\delta
}}(-|\nabla \xi _{i,\delta }|^{p_{i}(x)-2}\nabla \xi _{i,\delta }\nabla
\varphi _{i})\text{ }dx.%
\end{array}
\label{11}
\end{equation}%
Using (\ref{9}) we obtain 
\begin{equation}
\begin{array}{l}
\int_{\Omega }|\nabla \underline{u}_{i}|^{p_{i}(x)-2}\nabla \underline{u}%
_{i}\nabla \varphi _{i}\text{ }dx \\ 
=C^{-(p_{i}(x_{i}^{1})-1)}\int_{\Omega \backslash \overline{\Omega }_{\delta
}}\varphi _{i}\text{ }dx-C^{-(p_{i}(x_{i}^{2})-1)}\int_{\Omega _{\delta
}}\varphi _{i}\text{ }dx \\ 
\leq C^{-(p_{i}^{-}-1)}\int_{\Omega \backslash \overline{\Omega }_{\delta
}}\varphi _{i}\text{ }dx-C^{-(p_{i}^{+}-1)}\int_{\Omega _{\delta }}\varphi
_{i}\text{ }dx.%
\end{array}
\label{12}
\end{equation}%
Again, Theorem \ref{MVT2} and (\ref{9*}) imply%
\begin{equation}
\begin{array}{l}
\int_{\Omega }|\nabla \overline{u}_{i}|^{p_{i}(x)-2}\nabla \overline{u}%
_{i}\nabla \varphi _{i}\text{ }dx=\int_{\Omega }C^{p_{i}(x)-1}|\nabla \xi
_{i}|^{p_{i}(x)-2}\nabla \xi _{i}\nabla \varphi _{i}\text{ }dx \\ 
=C^{p_{i}(x_{i}^{0})-1}\int_{\Omega }|\nabla \xi _{i}|^{p_{i}(x)-2}\nabla
\xi _{i}\nabla \varphi _{i}\text{ }dx \\ 
\geq C^{p_{i}^{-}-1}\int_{\Omega }|\nabla \xi _{i}|^{p_{i}(x)-2}\nabla \xi
_{i}\nabla \varphi _{i}\text{ }dx=C^{p_{i}^{-}-1}\int_{\Omega }\varphi _{i}%
\text{ }dx,%
\end{array}
\label{13}
\end{equation}%
for certain $x_{i}^{0}\in \Omega $. Then, combining (\ref{12})-(\ref{13})
together implies%
\begin{equation*}
\begin{array}{l}
\int_{\Omega }|\nabla \underline{u}_{i}|^{p_{i}(x)-2}\nabla \underline{u}%
_{i}\nabla \varphi _{i}\text{ }dx\leq \int_{\Omega }|\nabla \overline{u}%
_{i}|^{p_{i}(x)-2}\nabla \overline{u}_{i}\nabla \varphi _{i}\text{ }dx,%
\end{array}%
\end{equation*}%
for all $\varphi _{i}\in W_{0}^{1,p_{i}(x)}(\Omega )$ with $\varphi _{i}\geq
0$, provided that $C>0$ is large enough. This proves the claim.

\mathstrut

Set%
\begin{equation}
\begin{array}{l}
R:=\underset{i=1,2}{\max }\{1,k_{p_{i}}\},%
\end{array}
\label{8}
\end{equation}%
where $k_{p_{1}}$ and $k_{p_{2}}$ are given by (\ref{5**}). The following
results allow us to achieve useful comparison properties.

\begin{proposition}
\label{P2} Assume $($\textrm{H}$_{\alpha ,\beta ,\gamma })$ is fulfilled
with $\alpha _{i}^{\mp }+\beta _{i}^{\mp }>0$ ($i=1,2$). Then, for $C>0$
large enough in (\ref{sub-super}), it holds 
\begin{equation}
-\Delta _{p_{i}(x)}\underline{u}_{i}\leq m_{i}\left\{ 
\begin{array}{ll}
\underline{u}_{1}^{\alpha _{i}(x)}\underline{u}_{2}^{\beta _{i}(x)} & \text{%
if }\alpha _{i}^{-},\beta _{i}^{-}>0 \\ 
\overline{u}_{1}^{\alpha _{i}(x)}\underline{u}_{2}^{\beta _{i}(x)} & \text{%
if }\alpha _{i}^{+}<0<\beta _{i}^{-} \\ 
\underline{u}_{1}^{\alpha _{i}(x)}\overline{u}_{2}^{\beta _{i}(x)} & \text{%
if }\beta _{i}^{+}<0<\alpha _{i}^{-}%
\end{array}%
\right. \text{ in }\Omega ,  \label{32}
\end{equation}%
\begin{equation}
-\Delta _{p_{i}(x)}\overline{u}_{i}\geq 2M_{i}(RC)^{\max \{\gamma _{i}^{+},%
\bar{\gamma}_{i}^{+}\}}+M_{i}\left\{ 
\begin{array}{ll}
\overline{u}_{1}^{\alpha _{i}(x)}\overline{u}_{2}^{\beta _{i}(x)} & \text{if 
}\alpha _{i}^{-},\beta _{i}^{-}>0 \\ 
\underline{u}_{1}^{\alpha _{i}(x)}\overline{u}_{2}^{\beta _{i}(x)} & \text{%
if }\alpha _{i}^{+}<0<\beta _{i}^{-} \\ 
\overline{u}_{1}^{\alpha _{i}(x)}\underline{u}_{2}^{\beta _{i}(x)} & \text{%
if }\beta _{i}^{+}<0<\alpha _{i}^{-}%
\end{array}%
\right. \text{ in }\Omega ,  \label{33}
\end{equation}%
where $R>0$ is provided in (\ref{8}), for $i=1,2.$
\end{proposition}

\begin{proof}
Assume $\alpha _{i}^{-},\beta _{i}^{-}>0$. From (\ref{sub-super}) and Lemma %
\ref{L5}, we have%
\begin{equation}
\begin{array}{l}
m_{i}\int_{\Omega }\underline{u}_{1}^{\alpha _{i}(x)}\underline{u}%
_{2}^{\beta _{i}(x)}\varphi _{i}\text{ }\mathrm{d}x=m_{i}\int_{\Omega
}C^{-(\alpha _{i}(x)+\beta _{i}(x))}\xi _{1,\delta }^{\alpha _{i}(x)}\xi
_{2,\delta }^{\beta _{i}(x)}\varphi _{i}\text{ }\mathrm{d}x \\ 
\geq m_{i}\int_{\Omega }(Cc_{0}^{-1})^{-(\alpha _{i}(x)+\beta
_{i}(x))}d(x)^{\alpha _{i}(x)+\beta _{i}(x)}\varphi _{i}\text{ }\mathrm{d}x
\\ 
\geq m_{i}(Cc_{0}^{-1})^{-(\alpha _{i}^{+}+\beta _{i}^{+})}\left( \delta
^{\alpha _{i}^{+}+\beta _{i}^{+}}\int_{\Omega \backslash \overline{\Omega }%
_{\delta }}\varphi _{i}\text{ }\mathrm{d}x+\int_{\Omega _{\delta
}}d(x)^{\alpha _{i}(x)+\beta _{i}(x)}\varphi _{i}\text{ }\mathrm{d}x\right)
\\ 
\geq C^{-(p_{i}^{-}-1)}\int_{\Omega \backslash \overline{\Omega }_{\delta
}}\varphi _{i}\text{ }\mathrm{d}x-C^{-(p_{i}^{+}-1)}\int_{\Omega _{\delta
}}\varphi _{i}\text{ }\mathrm{d}x,%
\end{array}
\label{16}
\end{equation}%
for all $\varphi _{i}\in W_{0}^{1,p_{i}(x)}(\Omega )$ with $\varphi _{i}\geq
0,$ $i=1,2,$ and for $C>0$ large enough. Thus, combining (\ref{12}) together
with (\ref{16}), we infer that%
\begin{equation*}
\int_{\Omega }|\nabla \underline{u}_{i}|^{p_{i}(x)-2}\nabla \underline{u}%
_{i}\nabla \varphi _{i}\text{ }\mathrm{d}x\leq m_{i}\int_{\Omega }\underline{%
u}_{1}^{\alpha _{i}(x)}\underline{u}_{2}^{\beta _{i}(x)}\varphi _{i}\text{ }%
\mathrm{d}x,
\end{equation*}%
for all $\varphi _{i}\in W_{0}^{1,p_{i}(x)}(\Omega )$ with $\varphi _{i}\geq
0,$ for $i=1,2$. This proves the first case in (\ref{32}).

Next, we show (\ref{33}) for $\alpha _{i}^{-},\beta _{i}^{-}>0$. Using (\ref%
{sub-super}), (\ref{8}), (\ref{9*}) and (\ref{13}), it follows that%
\begin{equation*}
\begin{array}{l}
M_{i}\int_{\Omega }(\overline{u}_{1}^{\alpha _{i}(x)}\overline{u}_{2}^{\beta
_{i}(x)}+2(RC)^{\max \{\gamma _{i}^{+},\bar{\gamma}_{i}^{+}\}})\varphi _{i}%
\text{ }\mathrm{d}x \\ 
=M_{i}\int_{\Omega }(C^{\alpha _{i}(x)+\beta _{i}(x)}\xi _{1}^{\alpha
_{i}(x)}\xi _{2}^{\beta _{i}(x)}+2(RC)^{\max \{\gamma _{i}^{+},\bar{\gamma}%
_{i}^{+}\}})\varphi _{i}\text{ }\mathrm{d}x \\ 
\leq M_{i}\int_{\Omega }(C^{\alpha _{i}^{+}+\beta _{i}^{+}}R^{\alpha
_{i}^{+}+\beta _{i}^{+}}+2(RC)^{\max \{\gamma _{i}^{+},\bar{\gamma}%
_{i}^{+}\}})\varphi _{i}\text{ }\mathrm{d}x \\ 
\leq \tilde{M}_{R}\max \{C^{\alpha _{i}^{+}+\beta _{i}^{+}},C^{\max \{\gamma
_{i}^{+},\bar{\gamma}_{i}^{+}\}}\}\int_{\Omega }\varphi _{i}\text{ }\mathrm{d%
}x \\ 
\leq C^{p_{i}^{-}-1}\int_{\Omega }\varphi _{i}\text{ }\mathrm{d}x\leq
\int_{\Omega }|\nabla \overline{u}_{i}|^{p_{i}(x)-2}\nabla \overline{u}%
_{i}\nabla \varphi _{i}\text{ }\mathrm{d}x,%
\end{array}%
\end{equation*}%
for $\varphi _{i}\in W_{0}^{1,p_{i}(x)}(\Omega )$ with $\varphi _{i}\geq 0,$ 
$i=1,2$, and for $C>0$ large enough.

Now, we deal with the other cases in (\ref{32}) and (\ref{33}) with respect
to the sign of the exponents. We only prove the inequalities corresponding
to the case $\alpha _{i}^{+}<0<\beta _{i}^{-}$ because the complementary
situation $\beta _{i}^{+}<0<\alpha _{i}^{-}$ is carried out in a similar
way. So assume $\alpha _{i}^{+}<0<\beta _{i}^{-}$. On account of Lemma \ref%
{L5} and $($\textrm{H}$_{\alpha ,\beta ,\gamma })$ one has%
\begin{equation*}
\begin{array}{l}
\int_{\Omega }C^{\alpha _{i}(x)-\beta _{i}(x)}\xi _{1}^{\alpha _{i}(x)}\xi
_{2,\delta }^{\beta _{i}(x)}\varphi _{i}\text{ }\mathrm{d}x\geq C^{\alpha
_{i}^{-}-\beta _{i}^{+}}R^{\alpha _{i}^{-}}(c_{0}\delta )^{\beta
_{i}^{+}}\int_{\Omega \backslash \overline{\Omega }_{\delta }}\varphi _{i}%
\text{ }\mathrm{d}x \\ 
\text{ \ \ \ \ \ \ \ \ \ }+C^{\alpha _{i}^{-}-\beta _{i}^{+}}c_{0}^{\beta
_{i}^{+}}c_{1}^{\alpha _{i}^{-}}\int_{\Omega _{\delta }}d(x)^{\alpha
_{i}(x)+\beta _{i}(x)}\varphi _{i}\text{ }\mathrm{d}x \\ 
\geq C^{-(p_{i}^{-}-1)}\int_{\Omega \backslash \overline{\Omega }_{\delta
}}\varphi _{i}\text{ }\mathrm{d}x-C^{-(p_{i}^{+}-1)}\int_{\Omega _{\delta
}}\varphi _{i}\text{ }\mathrm{d}x,%
\end{array}%
\end{equation*}%
for all $\varphi _{i}\in W_{0}^{1,p_{i}(x)}(\Omega )$ with $\varphi _{i}\geq
0$, provided that $C>0$ is large enough. Then, on the basis of (\ref%
{sub-super}), (\ref{9*}), (\ref{9}) and (\ref{12}), one gets%
\begin{equation*}
\begin{array}{l}
m_{i}\int_{\Omega }\overline{u}_{1}^{\alpha _{i}(x)}\underline{u}_{2}^{\beta
_{i}(x)}\varphi _{i}\text{ }\mathrm{d}x\geq \int_{\Omega }|\nabla \underline{%
u}_{i}|^{p_{i}(x)-2}\nabla \underline{u}_{i}\nabla \varphi _{i}\text{ }%
\mathrm{d}x.%
\end{array}%
\end{equation*}%
Next, we show (\ref{33}) when $\alpha _{i}^{+}<0<\beta _{i}^{-}$. By (\ref{8}%
), \textrm{(H}$_{\alpha ,\beta }$\textrm{)}, (\ref{6}), (\ref{5}) and Lemma %
\ref{L5}, it follows that%
\begin{equation}
\begin{array}{l}
M_{i}\int_{\Omega }(\underline{u}_{1}^{\alpha _{i}(x)}\overline{u}%
_{2}^{\beta _{i}(x)}+2(RC)^{\max \{\gamma _{i}^{+},\bar{\gamma}%
_{i}^{+}\}})\varphi _{i}\text{ }\mathrm{d}x \\ 
=M_{i}\int_{\Omega }(C^{-\alpha _{i}(x)+\beta _{i}(x)}\xi _{1,\delta
}^{\alpha _{i}(x)}\xi _{2}^{\beta _{i}(x)}+2(RC)^{\max \{\gamma _{i}^{+},%
\bar{\gamma}_{i}^{+}\}})\varphi _{i}\text{ }\mathrm{d}x \\ 
\leq M_{i}C^{-\alpha _{i}^{-}+\beta _{i}^{+}}\left[ (c_{0}\delta )^{\alpha
_{i}^{-}}R^{\beta _{i}^{+}}\int_{\Omega \backslash \overline{\Omega }%
_{\delta }}\varphi _{i}\text{ }\mathrm{d}x+c_{1}^{\beta
_{i}^{+}}c_{0}^{\alpha _{i}^{-}}\int_{\Omega _{\delta }}d(x)^{\alpha
_{i}^{-}+\beta _{i}^{+}}\varphi _{i}\text{ }\mathrm{d}x\right] \\ 
\text{ \ \ \ \ \ }+2(RC)^{\max \{\gamma _{i}^{+},\bar{\gamma}%
_{i}^{+}\}})\int_{\Omega }\varphi _{i}\text{ }\mathrm{d}x\leq
C^{p_{i}^{-}-1}\int_{\Omega }\varphi _{i}\text{ }\mathrm{d}x,%
\end{array}
\label{36}
\end{equation}%
for $\varphi _{i}\in W_{0}^{1,p_{i}(x)}(\Omega )$ with $\varphi _{i}\geq 0$,
provided that $C>0$ is large enough. Thus, gathering (\ref{13})-(\ref{36})
together yields%
\begin{equation*}
\begin{array}{l}
\int_{\Omega }|\nabla \overline{u}_{i}|^{p_{i}(x)-2}\nabla \overline{u}%
_{i}\nabla \varphi _{i}\text{ }\mathrm{d}x\geq M_{1}\int_{\Omega }(%
\underline{u}_{1}^{\alpha _{i}(x)}\overline{u}_{2}^{\beta
_{i}(x)}+2(RC)^{\max \{\gamma _{i}^{+},\bar{\gamma}_{i}^{+}\}})\varphi _{i}%
\text{ }\mathrm{d}x.%
\end{array}%
\end{equation*}
\end{proof}

\begin{proposition}
\label{P6} Assume $($\textrm{\~{H}}$_{\alpha ,\beta ,\gamma })$ is
fulfilled. Then, for $C>0$ large enough in (\ref{8}), it holds 
\begin{equation}
-\Delta _{p_{i}(x)}\underline{u}_{i}\leq m_{i}\left\{ 
\begin{array}{ll}
L^{\alpha _{i}^{-}}\underline{u}_{2}^{\beta _{i}(x)} & \text{if }\alpha
_{i}^{+}<0<\beta _{i}^{-} \\ 
L^{\beta _{i}^{-}}\underline{u}_{1}^{\alpha _{i}(x)} & \text{if }\beta
_{i}^{+}<0<\alpha _{i}^{-} \\ 
L^{\alpha _{i}^{-}+\beta _{i}^{-}} & \text{if }\alpha _{i}^{+},\beta
_{i}^{+}<0%
\end{array}%
\right. \text{ in }\Omega ,\text{ for }i=1,2,
\end{equation}%
where the constant $L>1$ is provided by Lemma \ref{L3}.
\end{proposition}

\begin{proof}
Assume $\alpha _{i}^{+}<0<\beta _{i}^{-}$. The case $\beta _{i}^{+}<0<\alpha
_{i}^{-}$ can be handled in much the same way. By (\ref{sub-super}), (\ref{5}%
) and (\ref{12}), one has%
\begin{equation*}
\begin{array}{l}
m_{i}\int_{\Omega }L^{\alpha _{i}^{-}}\underline{u}_{2}^{\beta
_{i}(x)}\varphi _{i}dx=m_{i}L^{\alpha _{i}^{-}}\int_{\Omega }(C^{-1}\xi
_{2,\delta })^{\beta _{i}(x)}\varphi _{i}dx \\ 
\geq m_{i}L^{\alpha _{i}^{-}}C^{-\beta _{i}^{+}}\int_{\Omega
}(c_{0}d(x))^{\beta _{i}(x)}\varphi _{i}dx \\ 
\geq m_{i}L^{\alpha _{i}^{-}}C^{-\beta _{i}^{+}}((c_{0}\delta )^{\beta
_{i}^{+}}\int_{\Omega \backslash \overline{\Omega }_{\delta }}\varphi
_{i}dx+\int_{\Omega _{\delta }}(c_{0}d(x))^{\beta _{i}(x)}\varphi _{i}dx) \\ 
\geq \int_{\Omega }\left\vert \nabla \underline{u}_{i}\right\vert
^{p_{i}(x)-2}\nabla \underline{u}_{i}\nabla \varphi _{i}dx,%
\end{array}%
\end{equation*}%
for all $\varphi _{i}\in W_{0}^{1,p_{i}(x)}(\Omega )$ with $\varphi _{i}\geq
0,$ and for $C>0$ large enough. If $\alpha _{i}^{+},\beta _{i}^{+}<0$, from (%
\ref{12}), it follows that%
\begin{equation*}
\begin{array}{l}
m_{i}\int_{\Omega }L^{\alpha _{i}^{-}+\beta _{i}^{-}}\varphi _{i}dx=
m_{i}L^{\alpha _{i}^{-}+\beta _{i}^{-}}\left( \int_{\Omega \backslash 
\overline{\Omega }_{\delta }}\varphi _{i}dx+\int_{\Omega _{\delta }}\varphi
_{i}dx\right) \\ 
\geq \int_{\Omega }\left\vert \nabla \underline{u}_{i}\right\vert
^{p_{i}(x)-2}\nabla \underline{u}_{i}\nabla \phi _{i}dx,%
\end{array}%
\end{equation*}%
for all $\varphi _{i}\in W_{0}^{1,p_{i}(x)}(\Omega )$ with $\varphi _{i}\geq
0,$ provided $C>0$ is sufficiently large. This ends the proof.
\end{proof}

\section{Proof of the main result}

\label{S5}

\subsection{Case $\protect\alpha _{i}^{\mp }+\protect\beta _{i}^{\mp }>0$}

Using the functions in (\ref{sub-super}) as well as the constant $R>0$ in (%
\ref{8}), we introduce the closed, bounded and convex set 
\begin{equation*}
\mathcal{K}_{C}=\left\{ (y_{1},y_{2})\in C_{0}^{1}(\overline{\Omega })^{2}:%
\underline{u}_{i}\leq y_{i}\leq \overline{u}_{i}\text{ in }\Omega \text{ \
and \ }\left\Vert \nabla y_{i}\right\Vert _{\infty }\leq CR\right\} .
\end{equation*}%
Define the map%
\begin{equation*}
\begin{array}{lll}
\mathcal{T}: & \mathcal{K}_{C} & \rightarrow C_{0}^{1}(\overline{\Omega }%
)\times C_{0}^{1}(\overline{\Omega }) \\ 
& (z_{1},z_{2}) & \mapsto \mathcal{T}%
(z_{1},z_{2})=(u_{1},u_{2})_{(z_{1},z_{2})},%
\end{array}%
\end{equation*}%
where $(u_{1},u_{2})$ is required to satisfy $(\mathrm{P}_{(z_{1},z_{2})})$.
It is worth noting that solutions of problem $(\mathrm{P}_{(z_{1},z_{2})})$
coincide with the fixed points of the operator $\mathcal{T}$. To reach the
desired conclusion, we shall apply Schauder's fixed point theorem.

For $(z_{1},z_{2})\in \mathcal{K}_{C}$ we have%
\begin{equation*}
z_{1}^{\alpha _{i}(x)}z_{2}^{\beta _{i}(x)}\leq \left\{ 
\begin{array}{ll}
\overline{u}_{1}^{\alpha _{i}(x)}\overline{u}_{2}^{\beta _{i}(x)} & \text{if 
}\alpha _{i}^{-},\beta _{i}^{-}>0 \\ 
\underline{u}_{1}^{\alpha _{i}(x)}\overline{u}_{2}^{\beta _{i}(x)} & \text{%
if }\alpha _{i}^{+}<0<\beta _{i}^{-} \\ 
\overline{u}_{1}^{\alpha _{i}(x)}\underline{u}_{2}^{\beta _{i}(x)} & \text{%
if }\beta _{i}^{+}<0<\alpha _{i}^{-}.%
\end{array}%
\right.
\end{equation*}%
In $\Omega$, using (\ref{sub-super}) together with Lemma \ref{L5}, we obtain%
\begin{equation*}
\begin{array}{ll}
z_{1}^{\alpha _{i}(x)}z_{2}^{\beta _{i}(x)} & \leq \left\{ 
\begin{array}{cc}
C^{\alpha _{i}^{+}+\beta _{i}^{+}}d(x)^{\alpha _{i}(x)+\beta _{i}(x)} & 
\text{if }\alpha _{i}^{-},\beta _{i}^{-}>0 \\ 
C^{-\alpha _{i}^{-}+\beta _{i}^{+}}d(x)^{\alpha _{i}(x)+\beta _{i}(x)} & 
\text{if }\alpha _{i}^{+}<0<\beta _{i}^{-} \\ 
C^{\alpha _{i}^{+}-\beta _{i}^{-}}d(x)^{\alpha _{i}(x)+\beta _{i}(x)} & 
\text{if }\beta _{i}^{+}<0<\alpha _{i}^{-}%
\end{array}%
\right. \\ 
& \leq \left\{ 
\begin{array}{ll}
C^{\alpha _{i}^{+}+\beta _{i}^{+}}\mbox{diam}(\Omega)^{\alpha _{i}(x)+\beta
_{i}(x)} & \text{if }\alpha _{i}^{-},\beta _{i}^{-}>0 \\ 
C^{-\alpha _{i}^{-}+\beta _{i}^{+}}\mbox{diam}(\Omega)^{\alpha _{i}(x)+\beta
_{i}(x)} & \text{if }\alpha _{i}^{+}<0<\beta _{i}^{-} \\ 
C^{\alpha _{i}^{+}-\beta _{i}^{-}}\mbox{diam}(\Omega)^{\alpha _{i}(x)+\beta
_{i}(x)} & \text{if }\beta _{i}^{+}<0<\alpha _{i}^{-}.%
\end{array}%
\right.%
\end{array}%
\end{equation*}%
Thus, we derive from $($\textrm{\textbf{H}}$_{f}),$ (\ref{sub-super}) and
Lemma \ref{L5} the estimate%
\begin{equation}
\begin{array}{l}
|f_{i}(x,z_{1},z_{2},\nabla z_{1},\nabla z_{2})|\leq M_{i}(z_{1}^{\alpha
_{i}(x)}z_{2}^{\beta _{i}(x)}+\left\vert \nabla z_{1}\right\vert ^{\gamma
_{i}(x)}+\left\vert \nabla z_{2}\right\vert ^{\bar{\gamma}_{i}(x)}) \\ 
\leq M_{i}C^{|\alpha _{i}^{\pm }|+|\beta _{i}^{\pm }|}L_{0}+2M_{i}(CR)^{\max
\{\gamma _{i}^{+},\bar{\gamma}_{i}^{+}\}}\text{ in }\Omega ,%
\end{array}
\label{15}
\end{equation}%
where constant $L_{0}>0$ is independent of $C$.

Consequently, the unique solvability of $(u_{1},u_{2})$ in $(\mathrm{P}%
_{(z_{1},z_{2})}),$ which is readily derived from Minty Browder's Theorem
(see, e.g., \cite{BR}), guarantees that $\mathcal{T}$ is well defined.
Moreover, the regularity theory up to the boundary in \cite{F} yields $%
(u_{1},u_{2})\in C_{0}^{1,\tau }(\overline{\Omega })^{2}$ for certain $\tau
\in (0,1)$ and a constant $\hat{R}>0$ such that it holds%
\begin{equation}
\left\Vert u_{i}\right\Vert _{C^{1,\tau }(\overline{\Omega })}<\hat{R}.
\label{14}
\end{equation}

\begin{proposition}
\label{P1}$\mathcal{K}_{C}$ is invariant by the operator $\mathcal{T}$.
\end{proposition}

\begin{proof}
Using the fact that $z_{1},z_{2}\in \mathcal{K}_{C}$, it follows that%
\begin{equation*}
z_{1}^{\alpha _{i}(x)}z_{2}^{\beta _{i}(x)}\geq \left\{ 
\begin{array}{ll}
\underline{u}_{1}^{\alpha _{i}(x)}\underline{u}_{2}^{\beta _{i}(x)} & \text{%
if }\alpha _{i}^{-},\beta _{i}^{-}>0 \\ 
\overline{u}_{1}^{\alpha _{i}(x)}\underline{u}_{2}^{\beta _{i}(x)} & \text{%
if }\alpha _{i}^{+}<0<\beta _{i}^{-} \\ 
\underline{u}_{1}^{\alpha _{i}(x)}\overline{u}_{2}^{\beta _{i}(x)} & \text{%
if }\beta _{i}^{+}<0<\alpha _{i}^{-}%
\end{array}%
\right. \text{ in }\Omega .
\end{equation*}%
Then, bearing in mind $($\textrm{\textbf{H}}$_{f})$ and Proposition \ref{P2}%
, the weak comparison principle entails%
\begin{equation}
\underline{u}_{1}\leq y_{1}\leq \overline{u}_{1}\text{ \ and \ }\underline{u}%
_{2}\leq u_{2}\leq \overline{u}_{2}\text{ in }\Omega .  \label{17}
\end{equation}%
On the other hand, since $\max \{|\alpha _{i}^{\pm }|+|\beta _{i}^{\pm
}|,\gamma _{i}^{+},\bar{\gamma}_{i}^{+}\}<p_{i}^{-}-1$, it follows from (\ref%
{15}) that%
\begin{equation*}
\begin{array}{l}
|f_{i}(x,z_{1},z_{2},\nabla z_{1},\nabla z_{2})|\leq (CR)^{p_{i}^{-}-1},%
\end{array}%
\end{equation*}%
provided that $C$ is sufficiently large. Hence, thanks to Lemma \ref{L2}, we
infer that%
\begin{equation}
\begin{array}{l}
\left\Vert \nabla u_{1}\right\Vert _{\infty },\left\Vert \nabla
u_{2}\right\Vert _{\infty }\leq CR.%
\end{array}
\label{18}
\end{equation}%
Consequently, gathering (\ref{14})-(\ref{18}) together yields $%
(u_{1},u_{2})\in \mathcal{K}_{C},$ showing that $\mathcal{T}(\mathcal{K}%
_{C})\subset \mathcal{K}_{C}$.
\end{proof}

\begin{proposition}
\label{P3}$\mathcal{T}$ is compact and continuous.
\end{proposition}

\begin{proof}
On the basis of (\ref{14}) and the compactness of the embedding $C^{1,\tau }(%
\overline{\Omega })\subset C_{0}^{1}(\overline{\Omega })$ we infer that $%
\mathcal{T(}C_{0}^{1}(\overline{\Omega })\times C_{0}^{1}(\overline{\Omega }%
))$ is a relatively compact subset of $C_{0}^{1}(\overline{\Omega })\times
C_{0}^{1}(\overline{\Omega })$. This shows the compactness of the operator $%
\mathcal{T}$.

Next, we prove that $\mathcal{T}$ is continuous with respect to the topology
of $C_{0}^{1}(\overline{\Omega })\times C_{0}^{1}(\overline{\Omega })$. Let $%
(z_{1,n},z_{2,n})\rightarrow (z_{1},z_{2})$ in $C_{0}^{1}(\overline{\Omega }%
)\times C_{0}^{1}(\overline{\Omega })$ for all $n$. Denoting $\left(
u_{1,n},u_{2,n}\right) =\mathcal{T(}z_{1,n},z_{2,n})$, we have from (\ref{14}%
) that $\left( u_{1,n},u_{2,n}\right) \in C^{1,\tau }(\overline{\Omega }%
)\times C^{1,\tau }(\overline{\Omega })$. By Ascoli-Arzel\`{a} Theorem there
holds 
\begin{equation*}
\left( u_{1,n},u_{2,n}\right) \rightarrow (u_{1},u_{2})\text{ in }C_{0}^{1}(%
\overline{\Omega })\times C_{0}^{1}(\overline{\Omega }).
\end{equation*}%
On the other hand, for $z_{1},z_{2}\in \mathcal{K}_{C}$ one has%
\begin{equation*}
f_{i}(x,z_{1,n},z_{2,n},\nabla z_{1,n},\nabla z_{2,n})\rightarrow
f_{i}(x,z_{1},z_{2},\nabla z_{1},\nabla z_{2})\in W^{-1,p_{i}^{\prime
}(x)}(\Omega ).
\end{equation*}%
Thus, we conclude that $\mathcal{T}$ is continuous.
\end{proof}

\subsection{Case $\protect\alpha _{i}^{\mp }+\protect\beta _{i}^{\mp }\leq 0$%
}

Using the functions $\underline{u}_{i}$ ($i=1,2$) in (\ref{sub-super}), the $%
W^{1,p_{i}(x)}$- gradient estimate $\tilde{L}$ in Lemma \ref{weak_bound} as
well as the $L^{\infty }$-bound $L$ in Lemma \ref{L3}, let introduce the set 
\begin{equation*}
\mathcal{\tilde{K}}_{\tilde{L}}=\left\{ 
\begin{array}{c}
(y_{1},y_{2})\in \underset{i=1,2}{\prod}W_{0}^{1,p_{i}(x)}(\Omega ):%
\underline{u}_{i}\leq y_{i}\leq L\text{ in }\Omega \text{ }\text{\ and \ }%
\left\Vert \nabla y_{i}\right\Vert _{p_{i}(x)}\leq \tilde{L}%
\end{array}%
\right\} ,
\end{equation*}%
which is closed, bounded and convex in $W_{0}^{1,p_{1}(x)}(\Omega )\times
W_{0}^{1,p_{2}(x)}(\Omega )$. Define the operator 
\begin{equation*}
\begin{array}{lll}
\mathcal{\tilde{T}}: & \mathcal{\tilde{K}}_{\tilde{L}} & \rightarrow
W_{0}^{1,p_{1}(x)}(\overline{\Omega })\times W_{0}^{1,p_{2}(x)}(\overline{%
\Omega }) \\ 
& (z_{1},z_{2}) & \mapsto \mathcal{\tilde{T}}(z_{1},z_{2})=(u_{1},u_{2}),%
\end{array}%
\end{equation*}%
where $(u_{1},u_{2})$ is required to satisfy $(\mathrm{P}_{(z_{1},z_{2})})$.
On account of (\ref{5}), (\ref{sub-super}) and Lemma \ref{weak_bound}, we
deduce that $(u_{1},u_{2})$ is the unique solution of problem $(\mathrm{P}%
_{(z_{1},z_{2})})$. Then, the map $\mathcal{\tilde{T}}$ is well defined.

\begin{proposition}
\label{P4}The set $\mathcal{\tilde{K}}_{\tilde{L}}$ is invariant by the
operator $\mathcal{\tilde{T}}$.
\end{proposition}

\begin{proof}
For any $(z_{1},z_{2})\in \mathcal{\tilde{K}}_{\tilde{L}},$ combining $(%
\mathrm{H}_{f})$ with Proposition \ref{P6} we derive that $u_{i}\geq 
\underline{u}_{i}$ in $\Omega$ $(i=1,2)$. Moreover, Lemma \ref{L3} implies
that $u_{i}\leq L$ while Lemma \ref{weak_bound} ensures that there exists a
large constant $\tilde{L}>0$ such that $\left\Vert \nabla y_{i}\right\Vert
_{p_{i}(x)}\leq \tilde{L}$. Hence, $u_{i}\in \mathcal{\tilde{K}}_{\tilde{L}}$
establishing that $\mathcal{\tilde{T}}(\mathcal{\tilde{K}}_{\tilde{L}%
})\subset \mathcal{\tilde{K}}_{\tilde{L}}$.
\end{proof}

\begin{proposition}
\label{P5}The map $\mathcal{\tilde{T}}$ is compact and continuous.
\end{proposition}

\begin{proof}
Let $(z_{1,n},z_{2,n})\rightarrow (z_{1},z_{2})$ in $W_{0}^{1,p_{1}(x)}(%
\Omega )\times W_{0}^{1,p_{2}(x)}(\Omega )$, that is, 
\begin{equation*}
(z_{1,n},z_{2,n})\rightarrow (z_{1},z_{2})\text{ in }L^{p_{1}(x)}(\Omega
)\times L^{p_{2}(x)}(\Omega )
\end{equation*}%
and 
\begin{equation*}
(\nabla z_{1,n},\nabla z_{2,n})\rightarrow (\nabla z_{1},\nabla z_{2})\text{
in }\left( L^{p_{1}(x)}(\Omega )\right) ^{N}\times \left(
L^{p_{2}(x)}(\Omega )\right) ^{N},
\end{equation*}%
with $(z_{1,n},z_{2,n})\in \mathcal{\tilde{K}}_{\tilde{L}}$. According to 
\cite[Theorem 2.4]{FZ}, the whole sequences $(z_{1,n},z_{2,n})$ and $(\nabla
z_{1,n},\nabla z_{2,n})$ converge in measure to $(z_{1},z_{2})$ and $(\nabla
z_{1},\nabla z_{2}),$ respectively. Consequently, given that the function $%
f_{i}$ is of Carath\'{e}odory type, one can write 
\begin{equation*}
\begin{array}{l}
f_{i}\left( x,z_{1,n}(x),z_{2,n}(x),\nabla z_{1,n}(x),\nabla
z_{2,n}(x)\right) \\ 
\text{ \ \ \ \ \ \ \ \ \ \ \ \ \ \ }\longrightarrow f_{i}\left(
x,z_{1}(x),z_{2}(x),\nabla z_{1}(x),\nabla z_{2}(x)\right) \text{ \ for a.e. 
}x\in \Omega .%
\end{array}%
\end{equation*}%
Again, \cite[Theorem 2.4]{FZ} ensures that $(\nabla z_{1,n},\nabla z_{2,n})$
converges to $(\nabla z_{1},\nabla z_{2})$ in modular, that is 
\begin{equation*}
\underset{n\rightarrow \infty }{\lim }\rho _{p_{i}(x)}\left( \nabla
z_{i,n}-\nabla z_{i}\right) =0,
\end{equation*}%
or equivalently 
\begin{equation*}
(\left\vert \nabla z_{1,n}-\nabla z_{1}\right\vert ^{p_{1}(x)},\left\vert
\nabla z_{2,n}-\nabla z_{2}\right\vert ^{p_{2}(x)})\rightarrow 0\text{ \ in }%
L^{1}(\Omega )\times L^{1}(\Omega ).
\end{equation*}%
Then, there exists a subsequence $(\left\vert \nabla z_{1,n_{k}}-\nabla
z_{1}\right\vert ^{p_{1}(x)},\left\vert \nabla z_{2,n_{k}}-\nabla
z_{1}\right\vert ^{p_{2}(x)})$ and positive measurable functions $%
(g_{1},g_{2})\in L^{1}(\Omega )\times L^{1}(\Omega )$ such that 
\begin{equation}
\left\{ 
\begin{array}{l}
\left\vert \nabla z_{1,n_{k}}(x)-\nabla z_{1}(x)\right\vert ^{p_{1}(x)}\leq
g_{1}(x) \\ 
\left\vert \nabla z_{2,n_{k}}(x)-\nabla z_{2}(x)\right\vert ^{p_{2}(x)}\leq
g_{2}(x)%
\end{array}%
\right. \text{ for a.e. }x\in \Omega .  \label{51}
\end{equation}%
Here, it is worth noting that $z_{i}\in \lbrack \underline{u}_{i},L]$ since $%
z_{i,n}\in \lbrack \underline{u}_{i},L]$ . Moreover, from Lemma \ref{L5} and
(\ref{sub-super}) it holds%
\begin{equation}
\begin{array}{ll}
z_{1,n}^{\alpha _{i}(x)}z_{2,n}^{\beta _{i}(x)} & \leq \left\{ 
\begin{array}{ll}
\underline{u}_{1}^{\alpha _{i}(x)}\underline{u}_{2}^{\beta _{i}(x)} & \text{%
if }\alpha _{i}^{+},\beta _{i}^{+}<0 \\ 
\underline{u}_{1}^{\alpha _{i}(x)}L^{\beta _{i}(x)} & \text{if }\alpha
_{i}^{+}<0<\beta _{i}^{-} \\ 
L^{\alpha _{i}(x)}\underline{u}_{2}^{\beta _{i}(x)} & \text{if }\beta
_{i}^{+}<0<\alpha _{i}^{-}.%
\end{array}%
\right. \\ 
& \leq M_{0}\left\{ 
\begin{array}{ll}
d(x)^{\alpha _{i}(x)+\beta _{i}(x)} & \text{if }\alpha _{i}^{+},\beta
_{i}^{+}<0 \\ 
d(x)^{\alpha _{i}(x)} & \text{if }\alpha _{i}^{+}<0<\beta _{i}^{-} \\ 
d(x)^{\beta _{i}(x)} & \text{if }\beta _{i}^{+}<0<\alpha _{i}^{-},%
\end{array}%
\right.%
\end{array}
\label{52}
\end{equation}%
for certain positive constant $M_{0}:=M_{0}(L,C,c_{0},\alpha _{i},\beta
_{i}) $.

Assume that $\alpha _{i}^{\pm }+\beta _{i}^{\pm }<0$. On account of $($%
\textrm{\textbf{H}}$_{f})$ and (\ref{52}), we deduce from (\ref{51}) that 
\begin{equation*}
\begin{array}{l}
\left\vert f_{i}\left( x,z_{1,n},z_{2,n},\nabla z_{1,n},\nabla
z_{2,n}\right) \right\vert \leq M_{i}(z_{1,n}^{\alpha _{i}(x)}z_{2,n}^{\beta
_{i}(x)}+\left\vert \nabla z_{1,n}\right\vert ^{\gamma _{i}(x)}+\left\vert
\nabla z_{2,n}\right\vert ^{\bar{\gamma}_{i}(x)}) \\ 
\leq M_{i}(K(x)+\left\{ g_{1}(x)^{\frac{1}{p_{1}(x)}}+|\nabla z_{1}|\right\}
^{\gamma _{i}(x)}+\left\{ g_{2}(x)^{\frac{1}{p_{2}(x)}}+|\nabla
z_{2}|\right\} ^{\overline{\gamma }_{i}(x)}).%
\end{array}%
\end{equation*}%
where 
\begin{equation*}
K(x)=\left\{ \begin{aligned} &(\tilde{c}d(x))^{\alpha _{i}(x)+\beta
_{i}(x)}&&\mbox{if }\alpha _{i}^{+},\beta _{i}^{+}<0\\
&(\tilde{c}d(x))^{\alpha
_{i}(x)}\left\|z_{2,n}^{\beta_i(x)}\right\|_{\infty}&&\mbox{if }\alpha
_{i}^{+}<0<\beta _{i}^{-}\\
&\left\|z_{1,n}^{\alpha_i(x)}\right\|_{\infty}(\tilde{c}d(x))^{\beta
_{i}(x)}&&\mbox{if }\beta_{i}^{+}<0<\alpha_{i}^{-}. \end{aligned}\right.
\end{equation*}%
Setting 
\begin{equation*}
\begin{array}{l}
G_{i}(x)=M_{i}(K(x)+\left\{ g_{1}(x)^{\frac{1}{p_{1}(x)}}+|\nabla
z_{1}|\right\} ^{\gamma _{i}(x)}+\left\{ g_{2}(x)^{\frac{1}{p_{2}(x)}%
}+|\nabla z_{2}|\right\} ^{\overline{\gamma }_{i}(x)}),%
\end{array}%
\end{equation*}%
we claim that $G_{i}\in L^{p_{i}^{\prime }(x)}(\Omega )$. Indeed, since $%
g_{1},g_{2}\in L^{1}(\Omega ),$ it is readily seen that
\begin{equation}
\begin{array}{l}
\left\{ g_{1}(x)^{\frac{1}{p_{1}(x)}}+|\nabla z_{1}|\right\} ^{\gamma
_{i}(x)}\in L^{\frac{p_{1}(x)}{\gamma _{i}(x)}}(\Omega )%
\end{array}
\label{56}
\end{equation}%
and%
\begin{equation}
\begin{array}{l}
\left\{ g_{2}(x)^{\frac{1}{p_{2}(x)}}+|\nabla z_{2}|\right\} ^{\overline{%
\gamma }_{i}(x)}\in L^{\frac{p_{2}(x)}{\overline{\gamma }_{i}(x)}}(\Omega ).%
\end{array}
\label{57}
\end{equation}%
Assumption $($\textrm{\~{H}}$_{\alpha ,\beta ,\gamma })$ ensures that the
embeddings $L^{\frac{p_{1}(x)}{\gamma _{i}(x)}}(\Omega )\hookrightarrow
L^{p_{i}^{\prime }(x)}(\Omega )$ and $L^{\frac{p_{2}(x)}{\overline{\gamma }%
_{i}(x)}}(\Omega )\hookrightarrow L^{p_{i}^{\prime }(x)}(\Omega )$ hold true
while $($\textrm{\~{H}}$_{\alpha ,\beta ,\gamma })$ together with an
argument similar to (\ref{lazer-mackenna}) guarantee that 
\begin{equation}
\int_{\Omega }K(x)^{p_{i}^{\prime }(x)}dx<\infty .  \label{58}
\end{equation}%
Then, gathering (\ref{56})-(\ref{58}) together we conclude that $G_{i}(x)\in
L^{p_{i}^{\prime }(x)}(\Omega ),$ showing the claim.

The generalized Lebesgue's dominated convergence theorem (see \cite[Lemma
3.2.8]{DHHR} ) implies that 
\begin{equation*}
f_{i}\left( x,z_{1,n_{k}},z_{2,n_{k}},\nabla z_{1,n_{k}},\nabla
z_{2,n_{k}}\right) \rightarrow f_{i}\left( x,z_{1},z_{2},\nabla z_{1},\nabla
z_{2}\right) \text{ \ in }L^{p_{i}^{\prime }(x)}(\Omega ).
\end{equation*}%
The convergence principle implies that the entire sequence $f_{i}\left(
x,z_{1,n},z_{2,n},\nabla z_{1,n},\nabla z_{2,n}\right) $ converges to $%
f_{i}\left( x,z_{1},z_{2},\nabla z_{1},\nabla z_{2}\right) $ in $%
L^{p_{i}^{\prime }(x)}(\Omega )\hookrightarrow W^{-1,p_{i}^{\prime
}(x)}(\Omega ),$ showing the continuity of $\mathcal{\tilde{T}}$.

Furthermore, it is worth noting that the operator $\mathcal{\tilde{T}}$ can
be written as 
\begin{equation*}
\mathcal{\tilde{T}}:=\left( \mathcal{L}_{1}^{-1}\circ \Phi _{1},\text{ }%
\mathcal{L}_{2}^{-1}\circ \Phi _{2}\right) ,
\end{equation*}%
where $\mathcal{L}_{i}=-\Delta _{p_{i}(x)}$ and $\Phi
_{i}(z_{1},z_{2})=f_{i}\left( x,z_{1},z_{2},\nabla z_{1},\nabla z_{2}\right) 
$ for all $(z_{1},z_{2})\in \tilde{\mathcal{K}}_{\tilde{L}}$. Thus, the
compactness of the embedding $L^{p_{i}^{\prime }(x)}(\Omega )\hookrightarrow
W^{-1,p_{i}^{\prime }(x)}(\Omega )$ implies that $\Phi _{i}\left( \mathcal{K}%
_{\Tilde{R}}\right) $ is a relatively compact subset of $W^{-1,p_{i}^{\prime
}(x)}(\Omega )$, hence the compactness of $\Phi _{i}$. Therefore, the
boundedness of $\mathcal{L}_{i}^{-1}$ (see \cite[Theorem 3.2]{KWZ}), leads
to the compactness of $\mathcal{\tilde{T}}$. The proof is completed.
\end{proof}

\subsection{Proof of Theorem \protect\ref{T1}}

By virtue of Propositions \ref{P1} and \ref{P3} (resp. Propositions \ref{P4}
and \ref{P5}), we are in a position to apply Schauder's fixed point theorem
(see, e.g., \cite{Z}) to the set $\mathcal{K}_{C}$ (resp. $\mathcal{\tilde{K}%
}_{\tilde{L}}$) and the map $\mathcal{T}:\mathcal{K}_{C}\rightarrow \mathcal{%
K}_{C}$ (resp. $\mathcal{\tilde{T}}:\mathcal{\tilde{K}}_{\tilde{L}%
}\rightarrow \mathcal{\tilde{K}}_{\tilde{L}}$). This ensures the existence
of $(u_{1},u_{2})\in \mathcal{K}_{C}$ (resp. $(u_{1},u_{2})\in \mathcal{%
\tilde{K}}_{\tilde{L}}$ satisfying $(u_{1},u_{2})=\mathcal{T}(u_{1},u_{2})$
(resp. $(u_{1},u_{2})=\mathcal{\tilde{T}}(u_{1},u_{2})$). Taking into
account the definition of $\mathcal{T}$ (resp. $\mathcal{\tilde{T}}$), it
turns out that $(u_{1},u_{2})\in C^{1,\tau }(\overline{\Omega })\times
C^{1,\tau }(\overline{\Omega })$ for certain $\tau \in (0,1)$ (resp. $%
(u_{1},u_{2})\in (W_{0}^{1,p_{1}(x)}(\Omega )\cap L^{\infty }(\Omega
))\times (W_{0}^{1,p_{2}(x)}(\Omega )\cap L^{\infty }(\Omega ))$) is a
(positive) solution of problem $($\textrm{P}$)$. Moreover, because the
solution $(u_{1},u_{2})$ lies in $\mathcal{K}_{C}$ (resp. $\mathcal{\tilde{K}%
}_{\tilde{L}}$), Lemma \ref{L5} implies that (\ref{44}) is fulfilled. This
completes the proof.

\section{Appendix}

\label{S6}

Denoting $B_{1}$ the unit ball and $S_{1}$ the unit sphere of $\mathbb{R}^{N}$, let $\Psi_{\epsilon }, \Pi_{\epsilon}:\mathbb{%
\mathbb{R}
}^{N}\rightarrow \mathbb{%
\mathbb{R}
}$ be sequences of mollifiers defined for $\epsilon >0$ by 
\begin{equation}
\Psi _{\epsilon }(x)=\frac{1}{\epsilon ^{N}}\Psi \left( \frac{x}{\epsilon }%
\right) \,\,\mbox{ and }\,\,\Pi _{\epsilon }(x)=\frac{1}{\epsilon ^{N}}\Pi \left( \frac{x}{\epsilon }%
\right),  \label{A1}
\end{equation}%
with $\Psi:B_{1}\rightarrow\mathbb{R}$ a bump function satisfying $\Psi(x)=C^{-1}e^{-\frac{1}{1-|x|^2}}$, $C=\int_{B_{1}}e^{-\frac{1}{1-|x|^2}}dx$, and $\Pi=\overline{C}^{-1}\pi$, where $\pi\in C_0^{1}(B_{1})$ is the solution of the problem
\begin{equation*}
    \left\{\begin{aligned}
     &-\Delta\pi=1 &&\mbox{ in }B_{1}\\
     &\pi=0 &&\mbox{ on }S_{1},
    \end{aligned}\right.
\end{equation*}
and $\overline{C}=\int_{B_{1}}\pi(x)dx$. Both the functions $\Psi$ and $\Pi$ satisfy 
\begin{equation}
\int_{\mathbb{\mathbb{R}}^{N}}\Psi (x)dx=\int_{\mathbb{\mathbb{R}}^{N}}\Pi(x)dx=1.\label{A2}
\end{equation}%
The Mean value theorem is stated as follows.

\begin{theorem}
\label{MVT2}Let $u\in W_{0}^{1,p(x)}(\Omega )$ be the solution of a
nonlinear elliptic equation of the form%
\begin{equation}
-\Delta _{p(x)}u=h(x)\text{ in }\Omega ,\text{ }u=0\text{ on }\partial
\Omega ,  \label{TVM}
\end{equation}%
where $h$ is a sign-constant function. Let $f:\Omega\rightarrow 
\mathbb{R}$ be a Lipschitz continuous function satisfying $-\infty< m\le
f(x)\le M<\infty $ for some constants $m, M$. Then, for any
sign-constant function $\phi \in W_{0}^{1,p(x)}(\Omega ),$ there exists a real 
$\gamma \in [m,M]$, depending on $\phi $, such that 
\begin{equation}
\int_{\Omega }f(x)|\nabla u|^{p(x)-2}\nabla u\nabla \phi dx=\gamma
\int_{\Omega }h(x)\phi dx.  \label{A8}
\end{equation}
\end{theorem}
\begin{remark}
Actually, equality (\ref{A8}) shall be proved for any nonnegative function $\phi\in C_0^{\infty}(\Omega)$, and the generalization to $W_0^{1,p(x)}(\Omega)$ is deduced by a density argument. Also, without loss of
generality, we will assume that $h$ is nonnegative, because otherwise, instead of (\ref{TVM}), we consider the problem 
\begin{equation*}
-\Delta _{p(x)}u=-h(x)\text{ in }\Omega ,\text{ \ }u=0\text{ on }\partial
\Omega ,
\end{equation*}%
whose solution is $\hat{u}=-u$. Similarly, we will assume that $\phi$ is nonnegative.
\end{remark}
\renewcommand*{\proofname}{Proof of Theorem \ref{MVT2}}
\begin{proof}Inspired by \cite{B}, the proof of (\ref{A8}) is divided in four steps:

$\bullet $ \textbf{Step 1: }For any functions $f,g:\mathbb{R}^N\rightarrow\mathbb{R}$, benote by $\star$ the classical convolution product defined by
\begin{align*}
    f\star g(x)=\int_{\mathbb{R}^N}f(x-z)g(z)dz,\,\,\mbox{for any }x\in\Omega,
\end{align*} and set
\begin{equation}
\begin{array}{l}
\phi_{\epsilon}=\phi\star\Pi_{\epsilon}\\
\Omega(y)=\left\{ x\in \Omega \,|\,f(x)\leq y\right\} ,\,\,y\in \left[m,M\right]\\
\Omega_{\epsilon}=\left\{ x\in \Omega \,|\,d(x,\partial\Omega)\geq \epsilon\right\}\\
a_{\epsilon,\Omega(y)}=\mathbbm{1}_{\Omega(y)\cap\Omega_{\epsilon}}\star\Psi_{\epsilon}.
\end{array}\label{A6}
\end{equation}%
For $\epsilon >0,$ let $\breve{F}_{\epsilon },F_{\epsilon }:[m,M]\rightarrow 
\mathbb{R}$ be the functionals defined for any nonnegative function $\phi
\in C_0^{\infty}(\Omega)$ by 
\begin{equation*}
\begin{array}{l}
\breve{F}_{\epsilon }(y)=\int_{\Omega}a_{\epsilon,\Omega(y)}(x)\,f(x)|\nabla u|^{p(x)-2}\nabla u\nabla
\phi_{\epsilon} dx, \\ 
F_{\epsilon }(y)=\int_{\Omega }a_{\epsilon,\Omega(y)}(x)\,|\nabla u|^{p(x)-2}\nabla u\nabla\phi_{\epsilon} dx,%
\end{array}
\end{equation*}%
where $u$ refers to the solution of problem (\ref{TVM}). We claim that 
\begin{equation}
\underset{\epsilon \rightarrow 0}{\lim }\left\vert \int_{m}^{M}d\breve{F}%
_{\epsilon }(y)-\int_{m}^{M}ydF_{\epsilon }(y)\right\vert =0.  \label{A10}
\end{equation}%
Indeed, let $P_{n}=\left\{ m=y_{0}<...<y_{n}=M\right\} $ be a partition of
the interval $\left[ m,M\right] $ such that $\mathcal{N}=y_{k}-y_{k-1}=%
\delta <\eta ,$ and let $y_{k}^{\prime }\in \lbrack y_{k-1},y_{k}]$. We
check that 
\begin{equation}
\begin{array}{l}
\left\vert \sum_{k=1}^{n}\left[ \breve{F}_{\epsilon }(y_{k})-\breve{F}%
_{\epsilon }(y_{k-1})\right] -\sum_{k=1}^{n}y_{k}^{\prime }\left[
F_{\epsilon }(y_{k})-F_{\epsilon }(y_{k-1})\right] \right\vert \\ 
=\left\vert \sum_{k=1}^{n}\int_{\Omega }(f(x)-y_{k}^{\prime})a_{\epsilon,\Omega_k}(x)|\nabla u|^{
p(x)-2}\nabla u\nabla\phi_{\epsilon} dx\right\vert .%
\end{array}
\label{A11}
\end{equation}%
with $\Omega_{k}=\{x\in\Omega\,|\,y_{k-1}<f(x)\le
y_k\}=\Omega(y_k)\backslash\Omega(y_{k-1})$ and $a_{\epsilon,\Omega_k}=%
a_{\epsilon,\Omega(y_{k})}-a_{\epsilon,\Omega(y_{k-1})}$. For any $x\in \Omega_{k}\cap \Omega_{\epsilon}+B_{\epsilon }(0) $, there exists $(y,s)\in
\Omega_{k}\cap \Omega_{\epsilon}\times B_{\epsilon }(0)$, such that $x=y+s$. From the Mean Value
theorem, and the $C$-Lipschitz regularity of $f$,
we infer that 
\begin{equation}
\begin{array}{l}
|f(x)-y_{k}^{\prime }|\le\left\vert f(y+s)-f(y)\right\vert +\left\vert
f(y)-y_{k}^{\prime }\right\vert\leq C|s|+|y_{k}-y_{k-1}|\leq C\epsilon
+\eta.  \end{array}\label{A12}
\end{equation}%
Since $a_{\epsilon,\Omega(y_{0})}=a_{\epsilon,\Omega(m)}=0$, and $%
a_{\epsilon,\Omega(y_{n})}=a_{\epsilon,\Omega(M)}=a_{\epsilon,\Omega}$%
, it holds 
\begin{equation}
\begin{array}{l}
\sum_{k=1}^{n}a_{\epsilon,\Omega_k}(x)=a_{\epsilon,\Omega}(x).%
\end{array}
\label{A13}
\end{equation}%
Gathering (\ref{A11})-(\ref{A13}) together leads to 
\begin{equation}
\begin{array}{l}
\left\vert \sum_{k=1}^{n}\left[ \breve{F}_{\epsilon }(y_{k})-\breve{F}%
_{\epsilon }(y_{k-1})\right] -\sum_{k=1}^{n}y_{k}^{\prime }\left[
F_{\epsilon }(y_{k})-F_{\epsilon }(y_{k-1})\right] \right\vert \\ 
\leq \left(C\epsilon +\eta \right)\int_{\Omega }a_{\epsilon,\Omega
}|\nabla
u|^{p(x)-1}\left\vert \nabla \phi_{\epsilon} \right\vert dx.
\end{array}
\label{A16}
\end{equation}%
Observe that 
\begin{equation}
\underset{\epsilon \rightarrow 0}{\lim }\,a_{\epsilon,\Omega(y)
}=\mathbbm{1}_{\Omega(y)
},\,\forall y\in[m,M],\,\,\mbox{ and }\,\,\underset{\epsilon \rightarrow 0}{\lim }\,\nabla \phi_{\epsilon}=\nabla \phi\,\,\mbox{ a.e. in }\Omega, \label{A18}
\end{equation}%
see for example \cite[Theorem 7, page 714]{Evans}. Thus, (\ref{A16})-(\ref{A18}) together imply 
\begin{equation}
\begin{array}{l}
\underset{\epsilon \rightarrow 0}{\lim }\,\underset{\eta \rightarrow 0}{\lim 
}\,\left\vert \sum_{k=1}^{n}\left[ \breve{F}_{\epsilon }(y_{k})-\breve{F}%
_{\epsilon }(y_{k-1})\right] -\sum_{k=1}^{n}y_{k}^{\prime }\left[
F_{\epsilon }(y_{k})-F_{\epsilon }(y_{k-1})\right] \right\vert =0.%
\end{array}
\label{A19}
\end{equation}%
This achieves the proof of (\ref{A10}).

$\bullet $ \textbf{Step 2: }After integrating by parts, we may write that 
\begin{equation}
\begin{array}{l}
\int_{m}^{M}ydF_{\epsilon }(y)=\left[ MF_{\epsilon }(M)-mF_{\epsilon }(m)\right] -\int_{m}^{M}F_{\epsilon
}(y)dy,%
\end{array}
\label{A20}
\end{equation}%
which implies, according to (\ref{A10}), 
\begin{equation}
\begin{array}{l}
\underset{\epsilon \rightarrow 0}{\lim }\left\vert \left( \breve{F}%
_{\epsilon }(M)-\breve{F}_{\epsilon }(m)\right) -\left( MF_{\epsilon
}(M)-mF_{\epsilon }(m)-\int_{m}^{M}F_{\epsilon }(y)dy\right) \right\vert =0.%
\end{array}
\label{A21}
\end{equation}%
By definition of $F_{\epsilon }$ and $\breve{F}_{\epsilon }$, (\ref{A21})
becomes 
\begin{equation}
\begin{array}{l}
\underset{\epsilon \rightarrow 0}{\lim }\left\vert \int_{\Omega }a_{\epsilon,\Omega}(x)\,f(x)|\nabla
u|^{p(x)-2}\nabla u\nabla\phi_{\epsilon} dx\right. \\ 
-M\,\int_{\Omega }a_{\epsilon,\Omega}(x)\,|\nabla u|^{p(x)-2}\nabla u\nabla \phi_{\epsilon} dx \\ 
\left. +\int_{m}^{M}\left[ \int_{\Omega }a_{\epsilon,\Omega(y)}(x)\,|\nabla u|^{p(x)-2}\nabla u\nabla
\phi_{\epsilon} dx\right] dy\right\vert =0.%
\end{array}
\label{A22}
\end{equation}%
By the way, $a_{\epsilon,\Omega(y)}=a_{\epsilon,\Omega}-a_{\epsilon,\Omega\backslash\Omega(y)}$, so from (%
\ref{A22}) we obtain 
\begin{equation}
\begin{array}{l}
\underset{\epsilon \rightarrow 0}{\lim }\left\vert \int_{\Omega }a_{\epsilon,\Omega}(x)\,f(x)|\nabla
u|^{p(x)-2}\nabla u\nabla \phi_{\epsilon} dx\right. \\ 
-m\,\int_{\Omega }a_{\epsilon,\Omega}(x)\,|\nabla u|^{p(x)-2}\nabla u\nabla \phi_{\epsilon} dx \\ 
\left. -\int_{m}^{M}\left[ \int_{\Omega }a_{\epsilon,\Omega\backslash\Omega(y)}(x)\,|\nabla u|^{p(x)-2}\nabla u\nabla
\phi_{\epsilon} dx\right] dy\right\vert =0.%
\end{array}
\label{A23}
\end{equation}%
$\bullet $ \textbf{Step 3: } Set $A(y)=\Omega(y)$ (resp. $A(y)=\Omega\backslash\Omega(y)$). Assume that $\phi\ge0$ is constant in $A(y)$. Then from the Dominated Convergence theorem and (\ref{A18}) we get
\begin{equation}\label{convergence}
    \begin{array}{l}
         \underset{\epsilon \rightarrow 0}{\lim }\int_{m}^{M}\left[ \int_{\Omega }a_{\epsilon,A(y)}(x)\,|\nabla u|^{p(x)-2}\nabla u\nabla
\phi_{\epsilon} dx\right] dy\\=\int_{m}^{M}\left[ \int_{\Omega }\mathbbm{1}_{A(y)}(x)\,|\nabla u|^{p(x)-2}\nabla u\nabla
\phi dx\right] dy=0.
    \end{array}
\end{equation}
Now suppose that $\phi$ is not constant in $A(y)$. For $(\epsilon,y,\phi)\in\mathbb{R}_{+}^{\star}\times[m,M]\times (C_0^{\infty}(\Omega))_{+}$ fixed, we claim that there exists a nonnegative function $\Phi_{\epsilon,y}\in W_0^{1,p(x)}(\Omega)$ such that 
\begin{align}\label{approx}
    \nabla\Phi_{\epsilon,y}=a_{\epsilon,A(y)}\nabla \phi_{\epsilon}\,\,\mbox{ a.e. in }\Omega.
\end{align}
 Indeed, consider the singular quasilinear elliptic problem
\begin{equation}\label{singular}
\left\{\begin{aligned}
&-\Delta v=-|a_{\epsilon,A(y)}\nabla\phi_{\epsilon}|^2v^{-1}-2div\left(a_{\epsilon,A(y)}\nabla\phi_{\epsilon}\right)&&\mbox{ in }\Omega\\
&v>0 &&\mbox{ in }\Omega\\
&v=0 &&\mbox{ on }\partial\Omega.
\end{aligned}\right.
\end{equation}%
Here, the divergence of the vector $a_{\epsilon,A(y)}\nabla\phi_{\epsilon}$ must be interpreted in the classical sense, which is possible since the convolution $\phi_{\epsilon}$ (resp. $a_{\epsilon,A(y)}$) involves the function $\phi$ (resp. $\Psi_{\epsilon}$) which is infinitely differentiable. We are going to prove the existence of at least one solution for (\ref{singular}) in the sense of distributions, i.e., there exists $v_{\epsilon}\in H_0^{1}(\Omega)$ such that 
\begin{equation*}
\begin{array}{l}
    \int_{\Omega}\nabla v_{\epsilon}\nabla\xi dx=-\int_{\Omega}|a_{\epsilon,A(y)}\nabla\phi_{\epsilon}|^2v_{\epsilon}^{-1}\xi dx-2\int_{\Omega}div\left(a_{\epsilon,A(y)}\nabla\phi_{\epsilon}\right)\xi dx
\end{array}
\end{equation*}
for every $\xi\in H_0^{1}(\Omega)$.\\
\textbf{Existence of a solution for (\ref{singular}): } Define
\begin{equation*}
\begin{array}{l}
a_{\epsilon}(x)=|a_{\epsilon,A(y)}\nabla\phi_{\epsilon}|\\
b_{\epsilon}(x)=div\left(a_{\epsilon,A(y)}\nabla\phi_{\epsilon}\right)\\
K_{\epsilon}=\left\{w\in C_0^1(\Omega):\underline{v}_{\epsilon}\le w\le \overline{v}_{\epsilon}\right\}\\
\underline{v}_{\epsilon}=\lambda\epsilon^{-1/2}\psi\\
\overline{v}_{\epsilon}= \max\{\lambda\epsilon^{-1/2}\|\psi\|_{\infty},\|b_{\epsilon}\|_{\infty}\|\chi\|_{\infty}\}\\
f_{\epsilon}(x,z)=-a_{\epsilon}^2(x)z^{-1}-2b_{\epsilon}(x),\,\,z\in K_{\epsilon}\\
\hat{\phi}(x)=\mathbbm{1}_{A(y)}(x)\underset{\epsilon\in(0,1)}{\min\,}\int_{B_1(0)}\phi(x-\epsilon y)dy
\end{array}
\end{equation*}
where $\psi,\chi\in C_0^{1}(\overline{\Omega})$ are respectively the solutions of the problems
\begin{equation}\label{a}
\left\{\begin{aligned}
&-\Delta\psi=\hat{\phi}&&\mbox{ in }\Omega\\
&\psi=0&&\mbox{ on }\partial\Omega
\end{aligned}\right.\,\,\mbox{ and }\,\,\left\{\begin{aligned}
&-\Delta\chi=1&&\mbox{ in }\Omega\\
     &\chi=0&&\mbox{ on }\partial\Omega,
\end{aligned}\right.
\end{equation}
see \cite[Proposition 2.1]{PS}. The fact that $\phi\ge0$ is not constant in $A(y)$ implies that $\hat{\phi}$ is a positive nontrivial function, hence the existence of $\lambda>0$ such that
\begin{equation}\label{SMP}
    d(x)\le \lambda\psi(x)\,\,\mbox{ for all }\,\,x\in\Omega,
\end{equation}
with $d(x):=d(x,\partial\Omega)$, see \cite[Theorems 1 and 2]{V}. Let $z\in K_{\epsilon}$, consider the auxiliary problem
\begin{align}\label{auxiliary}
    \left\{\begin{aligned}
    &-\Delta v=f_{\epsilon}(x,z)&&\mbox{ in }\Omega\\
&v=0&&\mbox{ on }\partial\Omega.
    \end{aligned}\right.
\end{align}
Following the ideas of \cite{PS}, we prove that $f_{\epsilon}\in L^{q}(\Omega)$. Indeed, due to the classical Mean Value theorem, the estimation $a_{\epsilon,A(y)}(x)\le\|\nabla a_{\epsilon,A(y)}\|_{\infty}d(x)$ with (\ref{SMP}) imply \begin{equation}\label{L^q}
\begin{array}{l}
     a_{\epsilon}^2z^{-1}\le a_{\epsilon}^2\underline{v}_{\epsilon}^{-1}= |a_{\epsilon,A(y)}\nabla\phi_{\epsilon}|^{2}(\lambda\epsilon^{-1/2}\psi)^{-1}\le\epsilon^{1/2}\|\nabla a_{\epsilon,A(y)}\|_{\infty}|\nabla\phi_{\epsilon}|^{2}\in L^{q}(\Omega),
\end{array}
\end{equation}
for any $q>N$. Consequently, $f_{\epsilon}\in L^{q}(\Omega)$, and from \cite[Proposition 2.1]{PS} there exists $\check{v}_{\epsilon}\equiv Az\in C_0^{1}(\Omega)$ satisfying (\ref{auxiliary}). Since the mapping $S:K_{\epsilon}\rightarrow L^{q}(\Omega)$, defined by $Sz=f_{\epsilon}(x,z)$, is continuous and bounded, it follows from \cite[Lemma 2.1]{Hai} that $A: K_{\epsilon}\rightarrow C_0^{1}(\Omega)$ is compact. It remains to prove that $A:K_{\epsilon}\rightarrow K_{\epsilon}$. Indeed, let $z\in K_{\epsilon}$
 and $\check{v}_{\epsilon}=Az$. Then 
     $-\Delta \check{v}_{\epsilon}\le\|b_{\epsilon}\|_{\infty}=-\Delta(\|b_{\epsilon}\|_{\infty}\chi)$, $\chi$ defined in (\ref{a}),
 which implies according to the weak comparison principle that 
     $\check{v}_{\epsilon}\le \|b_{\epsilon}\|_{\infty}\chi\le \overline{v}_{\epsilon}$. \\Now to prove that $\underline{v}_{\epsilon}\le\check{v}_{\epsilon}$, we proceed as in \cite{Hai}. Let $B_s=\{x\in\Omega:f_{\epsilon}(x,.)<\lambda s\hat{\phi}(x)\}$ ($s>0$) and let $v_{s,\epsilon}$ be the solution of the problem
 \begin{align*}
     -\Delta v= s^{-1}f_{s,\epsilon}(x,\underline{v}_{\epsilon})\equiv s^{-1}\left\{\begin{aligned}
     &f_{\epsilon}(x,\underline{v}_{\epsilon})\,\,&&\mbox{in }B_s\\
     &\lambda s\hat{\phi}(x)\,\,&&\mbox{in }\Omega\backslash B_s
     \end{aligned}\right., v=0\mbox{ on }\partial\Omega.
 \end{align*}
 Notice that 
     $-\Delta\check{v}_{\epsilon}=f_{\epsilon}(x,z)\ge f_{s,\epsilon}(x,\underline{v}_{\epsilon})=-\Delta\left(sv_{s,\epsilon}\right)$, and from the weak comparison principle we have $\check{v}_{\epsilon}\ge s v_{s,\epsilon}$. Also, setting $v_{0}=\lambda\psi$, $\psi$ defined in (\ref{a}), then $v_{0}$ satisfies $-\Delta v_{0}=\lambda\hat{\phi}$. From \cite[Lemma 2.1]{Hai}, and using (\ref{SMP}) and (\ref{L^q}), there exists $C>0$ independent of $v_{s,\epsilon}$ and $v_{0}$ such that
 \begin{equation}\label{B_s}
 \begin{array}{l}
     |v_{s,\epsilon}-v_{0}|_{C^1}\le C\left\| s^{-1}f_{s,\epsilon}-\lambda \hat{\phi}\right\|_{q}\\=C\left(\int_{B_s} \left|s^{-1}[-a_{\epsilon}^2(x)\underline{v}_{\epsilon}^{-1}-2b_{\epsilon}(x)]- \lambda \hat{\phi}(x)\right|^{q}dx\right)^{1/q}\\
     \le C|B_s|^{1/q}\left(s^{-1}\left[\epsilon^{1/2}\|\nabla a_{\epsilon,A(y)}\|_{\infty}\|\nabla\phi_{\epsilon}\|_{\infty}^{2}+2\|b_{\epsilon}\|_{\infty}\right]+\lambda\|\hat{\phi}\|_{\infty}\right)\\
     \le |B_s|^{1/q}M_{s,\epsilon},
 \end{array}
 \end{equation}
 where $$M_{s,\epsilon}=C(s^{-1}[\epsilon^{-1/2}\|\nabla\psi\|_{\infty}\|\nabla\phi\|_{\infty}^2+2\epsilon^{-1}\|\nabla\psi\|_{\infty}\|\nabla\phi\|_{\infty}+2\|\Delta\phi\|_{\infty}]+\lambda\|\phi\|_{\infty}|B_1|).$$
Also, from (\ref{SMP}), we check that
\begin{align}\label{lower_sol}
    \begin{array}{l}
         \check{v}_{\epsilon}\ge s v_{s,\epsilon}\ge \lambda s\left(\lambda^{-1} v_{0}-\lambda^{-1}|v_{s,\epsilon}-v_{0}|_{C^1}d(x)\right)\ge \lambda s(1-|v_{s,\epsilon}-v_{0}|_{C^1})\psi.
    \end{array}
\end{align}
 Now, consider the set $G=\{x\in\Omega:d(x,\partial [A(y)\cap\Omega_{\epsilon}])\ge\epsilon\}$. If $x\in G$, $b_{\epsilon}(x)=\mathbbm{1}_{A(y)}(x)\Delta\phi_{\epsilon}(x)$, and by definition of $\phi_{\epsilon}$ it occurs
\begin{equation}\label{term_0}
    \begin{array}{l}
-2b_{\epsilon}(x)=2\mathbbm{1}_{A(y)}(x)\phi\star(-\Delta\Pi_{\epsilon})(x)=2\epsilon^{-2}\mathbbm{1}_{A(y)}(x)\int_{B_{1}(0)}(-\Delta\Pi)(y)\phi(x-\epsilon y)dy\\=2\epsilon^{-2}\overline{C}^{-1}\mathbbm{1}_{A(y)}(x)\int_{B_{1}(0)}\phi(x-\epsilon y)dy.
    \end{array}
\end{equation}
Assuming that $\epsilon>0$ is small enough so that $\epsilon^{-1/2}\ge \lambda\overline{C}+2^{-1}N\|\nabla\phi\|_{\infty}\|\nabla\pi\|_{\infty} $, we deduce from (\ref{term_0}) that
\begin{equation}\label{term_1}
    \begin{array}{l}
    -2b_{\epsilon}(x)\ge(2\lambda+\overline{C}^{-1}N\|\nabla\phi\|_{\infty}\|\nabla\pi\|_{\infty}) \epsilon^{-3/2}\mathbbm{1}_{A(y)}(x)\int_{B_{1}(0)}\phi(x-\epsilon y)dy.
\end{array}
\end{equation}
Besides, considering that $\lambda>0$ satisfies (\ref{SMP}), for any $(x,z)\in G\times K_{\epsilon}$ we get
\begin{equation}\label{term_2}
    \begin{array}{l}
    -a_{\epsilon}^2(x)z^{-1}\ge-|a_{\epsilon,A(y)}\nabla\phi_{\epsilon}|^2\underline{v}_{\epsilon}^{-1}\ge-|a_{\epsilon,A(y)}\nabla\phi_{\epsilon}|^2(\epsilon^{-1/2}d(x))^{-1}\\\ge-\overline{C}^{-1}N\|\nabla\phi\|_{\infty}\|\nabla\pi\|_{\infty}\epsilon^{-3/2}\mathbbm{1}_{A(y)}(x)\int_{B_{1}(0)}\phi(x-\epsilon y)dy.
    \end{array}
\end{equation}
Therefore, combining (\ref{term_1}) and (\ref{term_2}), in the set $G$ we end up with
\begin{equation*}
\begin{array}{l}
    f_{\epsilon}(x,z)\ge 2\lambda\epsilon^{-3/2}\mathbbm{1}_{A(y)}(x)\int_{B_{1}(0)}\phi(x-\epsilon y)dy\ge\lambda s\hat{\phi}(x),
    \end{array}
\end{equation*}
with $s=2\epsilon^{-1/2}$. Consequently $B_{s}\subset\Omega\backslash G$, where $B_{s}=\{x\in\Omega: f_{\epsilon}(x,.)<\lambda s\hat{\phi}(x)\}$, and so $|B_s|\le|\Omega\backslash G|<k_{N}\epsilon^{N}$ for some constant $k_{N}>0$. Applying (\ref{B_s}) with $s=2\epsilon^{-1/2}$ and $q\in(N,(4/3)N)$ , we may write, for $\epsilon>0$ small enough,
\begin{equation}\label{1/2}
    |v_{s,\epsilon}-v_{0}|_{C^1}\le\epsilon^{3/4}M_{s,\epsilon}<\epsilon^{1/4}\overline{M}<1/2,
\end{equation}
where $\overline{M}>0$ is some constant independent of $\epsilon$. We deduce from (\ref{lower_sol}) and (\ref{1/2}) that \begin{equation*}
    \check{v}_{\epsilon}\ge \lambda(s/2)\psi=\lambda\epsilon^{-1/2}\psi=\underline{v}_{\epsilon}.
\end{equation*}
Therefore $\check{v}_{\epsilon}\in K_{\epsilon}$, and the existence of a solution of (\ref{singular}) follows from Schauder's fixed point theorem.\\
\textbf{Proof of (\ref{approx}): } Let $v_{\epsilon}\in K_{\epsilon}$ be a solution of (\ref{singular}). We show that 
\begin{align*}
    \int_{\Omega}\left|\nabla v_{\epsilon}-a_{\epsilon,A(y)}\nabla\phi_{\epsilon}\right|^2dx=0.
\end{align*}
 Indeed, multiplying (\ref{singular}) by $v_{\epsilon}$ and integrating over $\Omega$, we may write
\begin{equation*}
\begin{array}{l}
    0=\int_{\Omega}\left[|\nabla v_{\epsilon}|^{2}+\left|a_{\epsilon,A(y)}\nabla\phi_{\epsilon}\right|^2+2div(a_{\epsilon,A(y)}\nabla\phi_{\epsilon})v_{\epsilon}\right]dx\\
    = \int_{\Omega}\left[|\nabla v_{\epsilon}|^{2}-2a_{\epsilon,A(y)}\nabla\phi_{\epsilon}.\nabla v_{\epsilon}+\left|a_{\epsilon,A(y)}\nabla\phi_{\epsilon}\right|^2\right]dx\\
    =\int_{\Omega}\left|\nabla v_{\epsilon}-a_{\epsilon,A(y)}\nabla\phi_{\epsilon}\right|^2dx.
\end{array}
\end{equation*}
As a consequence, we define $\Phi_{\epsilon,y}\equiv v_{\epsilon}$ in (\ref{approx}).\\
$\bullet $ \textbf{Step 4: } Since $u$ is a solution of problem (\ref{TVM}), and thanks to the
positivity of the functions $h$ and $\Phi_{\epsilon,y}$, from (\ref{approx}) we have 
\begin{equation}\label{A24}
\begin{array}{l}
I_{\epsilon}=\int_{m}^{M}\left[ \int_{\Omega }a_{\epsilon,A(y)}(x)\,|\nabla u|^{p(x)-2}\nabla u\nabla \phi_{\epsilon} dx\right]
dy \\ 
=\int_{m}^{M}\left[ \int_{\Omega }|\nabla u|^{p(x)-2}\nabla u\nabla \Phi_{\epsilon,y} dx\right]
dy \\ 
=\int_{m}^{M}\left[ \int_{\Omega }h(x)\Phi_{\epsilon,y}dx\right] dy\geq 0.%
\end{array}
\end{equation}%
Besides, using the same argument as in (\ref{convergence}), we prove that the sequence $I_{\epsilon}$ converges as $\epsilon\rightarrow0$. Therefore from (\ref{A24}) we deduce that $\underset{\epsilon\rightarrow0}{\lim\,}I_{\epsilon}\ge0$, and combining (\ref{A22}) and (\ref{A23}), we end up
with 
\begin{equation*}
\begin{array}{l}
m\,\underset{\epsilon \rightarrow 0}{\lim }\int_{\Omega }a_{\epsilon,\Omega}(x)\,|\nabla
u|^{p(x)-2}\nabla u\nabla \phi_{\epsilon} dx \\ 
\leq \underset{\epsilon \rightarrow 0}{\lim }\int_{\Omega }a_{\epsilon,\Omega}(x)\,f(x)|\nabla
u|^{p(x)-2}\nabla u\nabla \phi_{\epsilon} dx \\ 
\leq M\,\underset{\epsilon \rightarrow 0}{\lim }\int_{\Omega }a_{\epsilon,\Omega}(x)\,|\nabla
u|^{p(x)-2}\nabla u\nabla\phi_{\epsilon} dx.%
\end{array}
\end{equation*}%
The Dominated Convergence theorem and (\ref{A18}) imply that
\begin{equation*}
\begin{array}{l}
m\int_{\Omega }h(x)\phi dx=m\int_{\Omega }|\nabla u|^{p(x)-2}\nabla u\nabla
\phi dx \\ 
\leq \int_{\Omega }f(x)|\nabla u|^{p(x)-2}\nabla u\nabla \phi dx \\ 
\leq M\int_{\Omega }|\nabla u|^{p(x)-2}\nabla u\nabla \phi dx=M\int_{\Omega
}h(x)\phi dx.%
\end{array}%
\end{equation*}%
This leads to the existence of $\gamma \in [m,M]$, depending on $\phi 
$, such that 
\begin{equation*}
\int_{\Omega }f(x)|\nabla u|^{p(x)-2}\nabla u\nabla \phi dx=\gamma
\int_{\Omega }h(x)\phi dx.
\end{equation*}%
This completes the proof.
\end{proof}

\end{document}